\documentclass{amsart}
\usepackage{amssymb,amsthm,amsmath}
\usepackage[english]{babel}
\usepackage[margin=3.0cm]{geometry}
\usepackage{enumerate}
\usepackage[T1]{fontenc}

\newtheorem{theorem}{Theorem}

\newtheorem{definition}{Definition}
\newtheorem{proposition}{Proposition}
\newtheorem{lemma}{Lemma}
\newtheorem{corollary}{Corollary}
\newtheorem{exam}{Example}
\newenvironment{example}{\begin{exam}\rm}{\end{exam}}
\newtheorem{exams}{Examples}

\newtheorem{rmk}{Remark}
\newenvironment{remark}{\begin{rmk}\rm}{\end{rmk}}
\newtheorem{notat}{Notation}
\newenvironment{notation}{\begin{notat}\rm}{\end{notat}}

\title[Partial separatrices and local Brunella's alternative]{Partial separatrices and local Brunella's alternative}
\author{F. Cano \& M. Ravara-Vago}

\address{Felipe Cano, Dep. Algebra, An. Mat. y Geom. y Top. - UVa -
Valladolid, Spain.}\email{fcano@agt.uva.es}

\address{Marianna Ravara-Vago, IRMAR, Universit\'e de Rennes 1 - Rennes, France.}\email{ravaravago@gmail.com}

\date{November 2014}
\begin{document}

\begin{abstract}
Here we state a conjecture concerning a local version of Brunella's alternative: any codimension one foliation in $({\mathbb C}^3,0)$ without germ of invariant surface has a neighborhood of the origin formed by leaves containing a germ of analytic curve at the origin.  We prove the conjecture for the class of codimension one foliations whose reduction of singularities is obtained by blowing-up points and curves of equireduction and such that the final singularities are free of saddle-nodes. The concept of ``partial separatrix'' for a given reduction of singularities has a central role in our argumentations, as well as the quantitative control of the generic Camacho-Sad index in dimension three. The ``nodal components'' are the only possible obstructions to get such germs of analytic curves. We use the partial separatrices to push the leaves near a nodal component towards compact diacritical divisors, finding in this way the desired analytic curves.
\end{abstract}

\maketitle

\tableofcontents

\section{Introduction}
In this paper we improve the main statements in \cite{Can-R-S} concerning a local version of Brunella's alternative for germs of codimension one holomorphic foliations.

 We know that any {\em non dicritical} germ of codimension one foliation $\mathcal F$ in $({\mathbb C}^3,0)$  always has an invariant germ of analytic surface, as  proved in \cite{Can-C} (the result is also true in higher ambient dimension \cite{Can-M}). Following a local version of Brunella's alternative \cite{Cer-L-R} and a conjecture of D. Cerveau \cite{Cer3}  we ask whether any  germ  of codimension one foliation ${\mathcal F}$ over $({\mathbb C}^3,0)$ without invariant germ of surface satisfies the following property:
 \begin{quote}
 \em $(\star)$
 There is an open neighborhood $U$ of \/ $0\in {\mathbb C}^3$ such that any leaf of ${\mathcal F}\vert_ U$ contains a germ of analytic curve at the origin.
 \end{quote}

 In view of the main result in \cite{Can}, any germ of codimension one foliation $\mathcal F$ in $({\mathbb C}^3,0)$ admits a reduction of singularities
\begin{equation*}
\label{eq:pi}
\pi:(M,\pi^{-1}(0))\rightarrow ({\mathbb C}^3,0).
\end{equation*}
 Using the arguments in \cite{Can-C}, we see that if $\mathcal F$ is without germ of invariant surface then
there is a {\em compact dicritical component} $D$ in the exceptional divisor $E$ of $\pi$ (this means that $D$ is an irreducible surface contained in $\pi^{-1}(0)$ and transversal to the transformed foliation $\pi^*{\mathcal F}$).

We develop our study inside the class of germs of codimension one foliations of {\em Complex Hyperbolic} type, for short CH-foliations. We recall \cite{Can-R-S} that a germ of codimension one foliation $\mathcal F$ in $({\mathbb C}^n,0)$ is a CH-foliation if for any generically transversal map
$$
\phi: ({\mathbb C}^2,0)\rightarrow ({\mathbb C}^n,0)
$$
the transformed foliation $\phi^*{\mathcal F}$ has no saddle-nodes in its reduction of singularities, that is $\phi^*{\mathcal F}$ is a {\em generalized curve} in the sense of \cite{Cam-N-S}. If $n=3$, given a reduction of singularities $\pi$ of $\mathcal F$, we have a CH-foliation if and only if there are no saddle-nodes among the singularities of $\pi^*{\mathcal F}$ of dimensional type two. We borrow the terminology of D. Cerveau in \cite{Cer2}, where ``complex hyperbolic'' stands for simple singularities in dimension two that are not saddle-nodes.

Although in this paper we only consider a particular class of codimension one foliations, we believe that there are enough reasons to state the following conjecture:
\begin{quote}
\em ``Any germ $\mathcal F$ of CH-foliation on $({\mathbb C}^3,0)$ without germ of invariant analytic surface satisfies $(\star)$''.
\end{quote}

Our general strategy to prove the conjecture is to show that all the leaves ``go'' to a compact dicritical component after  reduction of singularities. In fact, if $L$ is a leaf of $\pi^*{\mathcal F}$ intersecting a compact dicritical component $D$ at a point $p$, we can find a germ of analytic curve $(\tilde \gamma, p)\subset L$ and the image $(\pi(\tilde\gamma),0)$ is the desired germ of analytic curve. As we have shown in \cite{Can-R-S}, the main obstruction to following this strategy is the existence of a certain type of {\em uninterrupted nodal components}. They are a three-dimensional version of the ``nodal separators'' introduced by Mattei and Mar\'{\i}n in \cite{Mat-M}; they have also been recently considered by Camacho and Rosas \cite{Cam-R} in the study of local minimal invariant sets in dimension two. Now, the natural procedure is to prove that any {uninterrupted nodal component} goes to a compact dicritical component, carrying the leaves with it, and thus it does not produce an obstruction to property $(\star)$. Indeed, it is necessary to assume that the foliation has no invariant germ of surface. We interpret this fact after reduction of singularities by observing that all the {\em partial separatrices} also go to a compact dicritical component.

The relationship between uninterrupted nodal components and partial separatrices is the main argument we use in this paper to obtain a proof of the conjecture for a particular  class of CH-foliations on $({\mathbb C}^3,0)$.

Let us explain what are the {\em uninterrupted nodal components} and  the  {\em partial separatrices} for a
given reduction of singularities $\pi$ of a CH-foliation
$\mathcal F$ of $({\mathbb C}^3,0)$. First of all,
we quickly recall the final situation after reduction of singularities \cite{Can,Can-C}.

The exceptional divisor $E$ of $\pi$ is a normal crossings divisor and
the singular locus $\mbox{\rm Sing}\pi^*{\mathcal F}$ is a finite union of irreducible nonsingular curves having normal crossings with $E$. Any point  $p\in \mbox{\rm Sing}\pi^*{\mathcal F}$ has {\em dimensional type} $\tau_p\in \{2,3\}$, which corresponds to the number of variables needed to locally describe the foliation. 

If $\tau_p=2$, there are local coordinates $(x,y,z)$ at $p$ such that $\pi^*{\mathcal F}$ is given by
\begin{equation}
\label{eq:ttype}
\frac{dy}{y}-(\lambda +\phi(x,y))\frac{dx}{x}=0,\;  \phi(0,0)=0, \lambda\in {\mathbb C}\setminus {\mathbb Q}_{\geq 0}
\end{equation}
and moreover $(x=0)\subset E_{\mbox{\rm\small inv}}\subset (xy=0)$, where $E_{\mbox{\rm\small inv}}$ is the union of the invariant irreducible components of $E$. Note that $xy=0$ are invariant surfaces for $\mathcal F$ and that  the singular locus
$
\mbox{\rm Sing}\pi^*{\mathcal F}
$
 is
$
(x=y=0)
$
locally at $p$. The {\em transversal type} of $\pi^*{\mathcal F}$ at $p$ is the germ of foliation ${\mathcal T}_p$ in $({\mathbb C}^2,0)$ given by Equation (\ref{eq:ttype}).

Let $\Gamma$ be the only irreducible curve of
$
\mbox{\rm Sing}\pi^*{\mathcal F}
$ passing through $p$. We know that ${\mathcal T}_p={\mathcal T}_q$ for any $q\in \Gamma$ with $\tau_q=2$. Thus ${\mathcal T}_p={\mathcal T}_\Gamma$ is the {\em transversal type} of $\Gamma$. We say that $\Gamma$ is {\em  nodal} if $\lambda\in {\mathbb R}_{>0}$; in this case the transversal type is linearizable of the form $d(y/x^\lambda)=0$. If $\lambda\in {\mathbb R}_{<0}$, we say that $\Gamma$ is {\em a real saddle} and if $\lambda\in {\mathbb C}\setminus{\mathbb R}$ we say that $\Gamma$ is {\em a complex saddle}.

At a point $q$ of dimensional type three, the foliation $\pi^*{\mathcal F}$ is locally given by
$$
\frac{dx}{x}+(\lambda+\phi(x,y,z))\frac{dy}{y}+(\mu+\psi(x,y,z))\frac{dz}{z}=0
$$
where $\phi(0,0,0)=\psi(0,0,0)=0$ and
$
\lambda,\mu\in {\mathbb C}\setminus {\mathbb Q}_{\leq 0},\; \mu/\lambda \in {\mathbb C}\setminus {\mathbb Q}_{\leq 0}
$.
Moreover
$$
(xy=0)\subset E_{\mbox{\small\rm inv}}\subset (xyz=0).
$$
Note that the coordinate planes $xyz=0$ are invariant surfaces and
$$
\mbox{\rm Sing}\pi^*{\mathcal F}= (x=y=0)\cup (x=z=0)\cup (y=z=0).
$$
Thus there are exactly three curves $\Gamma_1,\Gamma_2,\Gamma_3$ of $\mbox{\rm Sing}\pi^*{\mathcal F}$ arriving at $q$. Up to reordering,  we have the following five possibilities:
\begin{enumerate}
\item $\Gamma_1,\Gamma_2$ are nodal curves and $\Gamma_3$ is  a real saddle.
\item $\Gamma_1$ is a nodal curve and $\Gamma_2,\Gamma_3$ are   complex saddles.
\item $\Gamma_1,\Gamma_2$ and $\Gamma_3$ are  real saddles.
\item $\Gamma_1$ is a real saddle and $\Gamma_2,\Gamma_3$ are complex saddles.
\item $\Gamma_1,\Gamma_2$ and $\Gamma_3$ are complex saddles.
\end{enumerate}
We define an {\em uninterrupted nodal component} ${\mathcal N}\subset \mbox{\rm Sing}\pi^*{\mathcal F}$ as any connected union of nodal curves such that at each point $q$ of dimensional type three there are exactly two curves $\Gamma_1,\Gamma_2\subset {\mathcal N}$ through $q$ (we have the first case in the  list above).  We say that  $\mathcal N$ is {\em incomplete} if it intersects the compact dicritical part of the exceptional divisor. As we have seen in  \cite{Can-R-S}, if  $\mathcal N$ is incomplete the leaves ``supported'' by $\mathcal N$ contain a germ of analytic curve. We  have also obtained the following result:
\begin{proposition}[\cite{Can-R-S}]
 \label{pro:starforincompletenodal}
 Consider a CH-foliation $\mathcal F$ on $({\mathbb C}^3,0)$ without germ of ana\-lytic surface and let $\pi$ be a reduction of singularities of $\mathcal F$. If any uninterrupted nodal component $\mathcal N$ is incomplete, then $\mathcal F$ satisfies $(\star)$.
\end{proposition}

Thus, the conjecture is proved once we assure that there is a reduction of singularities such that any uninterrupted nodal component is incomplete.

Let us now introduce the concept of {\em partial separatrix}. We say that a curve $\Gamma\subset\mbox{\rm Sing}{\pi^*{\mathcal F}}$ is a {\em trace curve} if it is contained in only one invariant irreducible component of the exceptional divisor $E$. Otherwise, the curve is the intersection of two invariant irreducible components of $E$ and it is a {\em corner curve}. By definition, a {\em partial separatrix} $C$ is any connected component of the union of trace curves. We say that  $C$ is {\em complete} if it does not intersect the compact dicritical part of $E$, otherwise, we say it is {\em incomplete}.

 Following Cano-Cerveau's argumentations as in \cite{Can-C}, given a partial separatrix $C$ we find a germ of invariant surface $$(S,C\cap \pi^{-1}(0))\subset (M,\pi^{-1}(0))$$ supported by $C$. The inclusion above is closed if and only if $C$ is complete. In this case we find by direct image a germ of surface $(\pi(S),0)$  invariant for $\mathcal F$. Hence, we conclude:
 \begin{quote}
 \em
 If $\mathcal F$ has no invariant germ of analytic surface, all the partial se\-paratrices are incomplete.
 \end{quote}
The incomplete partial separatrices are the ``guides'' we use to take the uninterrupted nodal components to a compact dicritical component of the exceptional divisor. To do this, we need an accurate control of the transitions of the Camacho-Sad indices along the curves in the singular locus from one component of the exceptional divisor to another. This quantitative analysis focused on the partial separatrices is in contrast with the qualitative and combinatorial arguments we used in \cite{Can-R-S} to obtain the first results concerning the conjecture.

 In this paper we prove the conjecture for the case of {\em special relatively isolated complex hyperbolic} germs  $\mathcal F$ of codimension one foliations in $({\mathbb C}^3,0)$. We precise the definitions in the next sections, but roughly speaking, this means that we can perform a reduction of singularities by blowing-up points until we reach a situation of equireduction along non compact curves, which we resolve by blowing-up only curves. This class of foliations contains both the cases of equireduction and the foliations associated to absolutely isolated singularities of surfaces. There are previous works on absolutely isolated singularities of vector fields \cite{Cam-C-S} or on foliations desingularized by punctual blow-ups \cite{Can-C-S};  also, the results of Sancho de Salas in \cite{San} concern these conditions very closely.

The main result of this paper is:
\begin{theorem}
 \label{teo:mainteo}
 Any special relatively isolated CH-foliation $\mathcal F$ in $({\mathbb C}^3,0)$ without germ of invariant ana\-lytic surface satisfies property $(\star)$.
 \end{theorem}

Theorem \ref{teo:mainteo} improves the results in \cite{Can-R-S}. What we know from \cite{Can-R-S} is that if we take a complete nodal component $\mathcal N$,  then the projection of $\mathcal N$ contains at least one of the germs of curve of $\mbox{\rm Sing}{\mathcal F}$ in  $({\mathbb C}^3,0)$. In this way we have a criterion for the non existence of complete uninterrupted nodal components by looking at generic points of the germs of curve in $\mbox{\rm Sing}{\mathcal F}$. 

We prove Theorem \ref{teo:mainteo} by showing that all the uninterrupted nodal components are incomplete. The argument is based on a control of the evolution of {\em incomplete points}. They are points such that there is a ``local'' partial separatrix over them which is incomplete. At the final step of the reduction of singularities, all the points are complete. We find a contradiction with the existence of a complete uninterrupted nodal component ${\mathcal N}$ as follows. At the ``birth level'' of ${\mathcal N}$ in the sequence of reduction of singularities, we find an incomplete point in a particular situation concerning the partial separatrices through it. We prove that this situation is part of a class of scenarios which persists along the reduction of singularities. In each scenario, there is at least one incomplete point. Then ``a fortiori'' we find an incomplete point at the last step of the reduction of singularities and obtain the desired contradiction.\\ 

\noindent{\em Acknowledgements}. The second author is very grateful to the group ECSING and the Department of Mathematics at Universidad de Valladolid for the warm welcome on the numerous visits.\\   

\noindent {\em Funding}. This work was partially supported by the Spanish research project MTM2010-15471 and the Brazilian PDE/CsF scholarship 245480/2012-9 from CNPq.

\section{Special Relatively Isolated CH-foliations}

A {\em special relatively isolated sequence ${\mathcal S}=\{\pi_k\}_{k=1}^N$ of blow-ups of $({\mathbb C}^3,0)$} is a sequence
$$
{\mathcal S}: ({\mathbb C}^3,0)=(M_0,F_0)\stackrel{\pi_1}{\leftarrow} (M_1,F_1) \stackrel{\pi_2}{\leftarrow}\cdots \stackrel{\pi_N}{\leftarrow} (M_N,F_N),\quad F_{k+1}=\pi_{k+1}^{-1}(F_k),
$$
given by blow-ups $\pi_{k+1}$ with center at closed germs
$(Y_{k},Y_{k}\cap F_k)\subset (M_{k},F_k)$,
such that for any $0\leq k\leq N-1$ we have
\begin{enumerate}
\item   $Y_{k}\cap F_k$ is a single point $Y_{k}\cap F_k=\{p_k\}$.
\item   $Y_{k}$ is either $\{p_k\}$ or a germ of nonsingular closed curve $(Y_k,p_k)\subset (M_k,F_k)$ having normal crossings with the exceptional divisor $E^{k}\subset M_{k}$ of
    $$
    \sigma_k:(M_k,F_k)\rightarrow ({\mathbb C}^3,0),
    $$
    where
    $\sigma_k=\pi_1\circ\pi_2\circ\cdots \circ\pi_k$.
    \item If $(Y_k,p_k)$ is a germ of curve, we have an {\em equireduction sequence over  $(Y_k,p_k)$} in the following sense. For any $k<\ell\leq N-1$ we have one of the following situations:
        \begin{enumerate}
        \item $p_k\notin \pi_{k,\ell}(Y_{\ell})$, where $ \pi_{k,\ell}=\pi_{k+1}\circ\pi_{k+2}\circ\cdots\circ \pi_\ell$.
        \item The center $(Y_\ell,p_\ell)$ is a germ of curve and $\pi_{k,\ell}$ induces an isomorphism $$\bar\pi_{k,\ell}:(Y_\ell,p_\ell)\rightarrow (Y_k,p_k). $$
        \end{enumerate}
\end{enumerate}
Now, we say that a CH-foliation $\mathcal F$ over $({\mathbb C}^{3},0)$ is a {\em special relatively isolated CH-foliation} if there is a sequence $\mathcal S$ of blow-ups as above  such that
\begin{enumerate}
\item[i)] For any $0\leq k\leq N-1$ the center $Y_{k}\subset M_{k}$ of the blow-up $\pi_{k+1}$ is contained in  the locus $\mbox{\rm Sing}^*({\mathcal F}_{k},E^{k})$ of non simple points for ${\mathcal F}_k, E^k$, where ${\mathcal F}_{k}=\sigma_k^*{\mathcal F}$ is the transform of $\mathcal F$  by $\sigma_k$.
\item[ii)] All the points in $F_N$ are simple points for ${\mathcal F}_N, E^N$. (See \cite{Can})
\end{enumerate}
Since $\mathcal F$ is a CH-foliation, the simple points after reduction of singularities are without saddle-nodes. Conversely, the fact that there is a reduction of singularities without saddle-nodes in the last step is enough to assure that $\mathcal F$ is a CH-foliation.
We refer to \cite{Can-R-S} for more details on the definitions.

Now, let us introduce some useful notations and remarks. The exceptional divisor $E^k\subset M_k$ of $\sigma_k$ is a union of components
$$
E^k=E^k_1\cup E^k_2\cup\cdots\cup E^k_k
$$
where $E^k_\ell$ is the stric transform of $E^{k-1}_\ell$ for $1\leq\ell<k$ and $E^k_k=\pi_{k}^{-1}(Y_k)$ is the exceptional divisor of the last blow-up $\pi_k$. Recall that $M_k$ is a germ over the fiber $F_k=\sigma_k^{-1}(0)$.
For any $k\geq 1$ we have that $F_k\subset E^k$ and the exceptional divisor $E^k$ is a closed germ
$$
(E^k,F_k)\subset (M_k,F_k).
$$
A given irreducible component $E^k_i$ may be an {\em invariant component}, if it is invariant for ${\mathcal F}_k$, or a {\em dicritical component}, when it is generically transversal to the foliation. We denote $E_{\mbox{\sl\small inv}}^k$ the union of the invariant components and   $E_{\mbox{\sl\small dic}}^k$ the union of the dicritical ones.

Recall the definition of $\pi_{k,\ell}=\pi_{k+1}\circ\pi_{k+2}\circ\cdots\circ \pi_\ell$. For certain special cases, we adopt the following simplified notations:
\begin{eqnarray*}
\sigma_k=\pi_{0,k}&:&(M_k,F_k)\rightarrow ({\mathbb C}^3,0)\\
\rho_k=\pi_{k,N}&:&(M,F)\rightarrow (M_k,F_k)\\
\pi=\pi_{0,N}&:& (M,F)\rightarrow ({\mathbb C}^3,0).
\end{eqnarray*}
We denote the final step of the reduction of singularities by $$M=M_N,\; E=E^N,\; \pi^*{\mathcal F}={\mathcal F}_N,\; F=F_N=\pi^{-1}(0).$$

From now on, we fix a special relatively isolated CH-foliation $\mathcal F$ of $({\mathbb C}^3,0)$ without germ of invariant analytic surface and a special relatively isolated sequence of blow-ups $\mathcal S$ performing a reduction of singularities of $\mathcal F$.

If the first blow-up $\pi_1$ is centered at the origin $0\in {\mathbb C}^3$, we have that the fiber $F_k$ is the union $E^k_{\mbox{\sl\small c}}$ of the compact components of $E^k$, for any $1\leq k\leq N$. As we shall explain later, the case when the first blow-up is centered in a germ of curve is of no interest to us, since in this situation the foliation $\mathcal F$ has invariant surfaces.
Thus, we  also suppose along the paper that $\pi_1$ is a blow-up centered at the origin and hence $F_k=E^k_{\mbox{\sl\small c}}$.

\section{Partial separatrices}
\label{sec:partialseparatrices}
Here we do a revision, adapted to our case,  of the partial separatrices introduced in \cite{Can-R-S} and we briefly recall the global picture of a reduction of singularities (see \cite{Can} for more details).

\begin{definition} A {\em partial separatrix $C$ for ${\mathcal F}, {\pi}$} is any connected component  of the union $T$ of the trace curves of $\mbox{\rm Sing}(\pi^*{\mathcal F})$. We say that $C$ is {\em complete} if it does not intersect the union $E_{\mbox{\sl\small c,dic}}$ of the compact dicritical components of $E$.
We say that $C$ is {\em incomplete} if it does intersect $E_{\mbox{\sl\small c,dic}}$.
\end{definition}
A partial separatrix $C$ must be considered as a connected component of the germ $(T,T\cap F)$. So it is also a germ $(C,C\cap F)$. For shortness, we write $C$ to denote the partial separatrix if there is no risk of confusion. The {\em compact part} $C\cap F$ of a partial separatrix  is also connected and it is the union of the compact curves in $C$ or just a single point.
\begin{example}
Let us recall Darboux-Jouanolou's example \cite{Jou}.  It is the conic foliation in $({\mathbb C}^3,0)$ given by the 1-form
$$
\omega= (z^{m+1}-x^my)dx+(x^{m+1}-y^{m}z)dy+(y^{m+1}-z^m x)dz.
$$
The reduction of singularities consists of an initial dicritical blow-up of the origin followed by $m^2+m+1$ blow-ups centered at each of the lines of the singular locus
$$
z^{m+1}-x^my= x^{m+1}-y^{m}z= y^{m+1}-z^m x=0.
$$
We find $2(m^2+m+1)$ partial separatrices, all of them  incomplete. Each one is a single non compact curve $(C^{(i)},p_i)$, $i=1,2,\ldots, 2(m^2+m+1)$  and hence the compact part $C^{(i)}\cap F$ is just the point $p_i$.
\end{example}

Let $C$ be a partial separatrix and take a point $p\in C\cap F$. Recalling that the final singularities are complex hyperbolic, depending on the dimensional type $\tau=\tau(\pi^*{\mathcal F};p)$ we find two situations:
\begin{enumerate}
\item If $\tau=2$, there are coordinates $x,y,z$ at $p$ such that $E_{\mbox{\sl\small inv}}=(x=0)$, $E_{\mbox{\sl\small dic}}\subset (z=0)$ and
\begin{eqnarray*}
  C=(x=y=0)=
  \mbox{\rm Sing}(\pi^{*}{\mathcal F}).
 \end{eqnarray*}
Moreover $S=(y=0)$ is the only invariant germ of surface for $\pi^*{\mathcal F}$ at $p$ not contained in $E$.
 \item If $\tau=3$, there are coordinates $x,y,z$ at $p$ such that $E=E_{\mbox{\sl\small inv}}=(xy=0)$,
 \begin{eqnarray*}
  C=(x=z=0)\cup (y=z=0),
 \end{eqnarray*}
and $\mbox{\rm Sing}(\pi^{*}{\mathcal F})=C\cup (x=y=0)$.
 Moreover  $S=(z=0)$ is the only invariant germ of surface for $\pi^*{\mathcal F}$ at $p$ not contained in $E$.
\end{enumerate}
Gluing these situations along $C$ as in \cite{Can-C}, we find a germ of surface $(S_C,C\cap F)$ invariant for $\pi^*{\mathcal F}$ and not contained in $E$. Moreover, the inclusion of germs
$$
(S_C,C\cap F)\subset (M,F)
$$
is a closed immersion if and only if $S_C\cap F=C\cap F$. On the other hand, we have
$$
S_C\cap F=C\cap F \Leftrightarrow C\cap E_{\mbox{\sl\small c,dic}}=\emptyset.
$$
That is, we obtain a closed immersion exactly when $C$ is a complete partial separatrix. In this case, by Grauert's Theorem of the direct image under a proper morphism, we obtain a germ of surface $(\pi(S_C),0)$ invariant for $\mathcal F$. We conclude:
\begin{proposition}
\label{pro:incompletitudpartialsep}
If $\mathcal F$ has no invariant germ of surface, then all the partial separatrices are incomplete.
\end{proposition}
To finish this section, we give a result that  justifies our assumption on the first blow-up being centered at the origin.
 \begin{proposition} If the first blow-up is centered at a germ of curve $\gamma$, then $\mathcal F$ has a germ of invariant surface.
\end{proposition}
\begin{proof} Note that we have equireduction along $\gamma$ and thus the fiber $F=\pi^{-1}(0)$ is a union of compact curves.
 We have the following possible cases:
\begin{enumerate}
\item $F$ contains a non invariant curve.
\item There is a dicritical component in the exceptional divisor, but all the curves in $F$ are invariant.
\item All the components of the exceptional divisor are invariant and there is a curve $\Gamma$ in $F$ contained in the singular locus of $\pi^*{\mathcal F}$.
\item All the components of the exceptional divisor are invariant and the curves in $F$ are not contained in the singular locus of $\pi^*{\mathcal F}$.
\end{enumerate}

If there is a non invariant curve $\Gamma\subset F$, then  $\Gamma$ is necessarily contained in a dicritical component of $E$. Taking a generic point $p\in \Gamma$, the foliation $\pi^*{\mathcal F}$ is non singular at $p$ and transversal to $\Gamma$. Thus, we find a germ of invariant surface $(\tilde S,p)$ that gives a closed immersion into $(M,F)$ and hence it projects onto
a germ of surface $(S,0)$ invariant for $\mathcal F$.

If all the  curves of $F$ are invariant but there is a dicritical component $E_i$ of $E$, we consider the curve $\Gamma=E_i\cap F$. At a generic point $p$ of $\Gamma$, there is a germ of surface $(\tilde S,p)$ not contained in $E$ such that $(\Gamma,p)=(\tilde S\cap E_i,p)$. By an extension of the argument of Cano-Cerveau \cite{Can-C} also used in   \cite{Reb-R} we can prolong $(\tilde S,p)$ over the fiber $F$  to find a closed immersion of a germ of invariant surface $(\tilde S,F)$ in $(M,F)$. Finally, we project it by $\pi$ to obtain a germ of surface $(S,0)$ invariant for $\mathcal F$. This argument is also valid for the case $(3)$.

Suppose now that all the irreducible components of the exceptional divisor $E$ are invariant and the curves in $F$ are not contained in the singular locus of $\pi^*{\mathcal F}$. This gives a non dicritical equireduction along $\gamma$ in the sense of \cite{Can2, Can-M}. In those papers it is proved that the reduction of singularities is given by the one of
 ${\mathcal F}\vert_{\Delta}$, where $\Delta$ is a plane transverse to $\gamma$. Then, any Camacho-Sad separatrix $\Sigma$ of ${\mathcal F}\vert_{\Delta}$ induces a germ of surface $(S,0)$ invariant for $\mathcal F$.
\end{proof}

\section{Partial separatrices at intermediate steps}
Let us give some remarks and definitions concerning the behavior of  partial separatrices  at an intermediate step $(M_k,F_k)$ of the sequence $\mathcal S$ of reduction of singularities, with $0\leq k\leq N$.
\begin{notation} If $C$ is a partial separatrix, we denote $C_k=\rho_{k}(C)$. Let us remark that $C_k\cap F_k$ is a connected nonempty compact set and $\rho_k(C\cap F)=C_k\cap F_k$.
\end{notation}
Consider an irreducible compact curve $\Gamma\subset M_k$  in the singular locus $\mbox{\rm Sing}{\mathcal F}_k$. We have that  $\Gamma\subset F_k$. By the properties of the sequence $\mathcal S$, only finitely many points of $\Gamma$ will be modified in the further blow-ups $\pi_{k+1},\pi_{k+2},\ldots,\pi_N$. Thus, there is a well defined strict transform $\Gamma'\subset M$ of $\Gamma$ under $\rho_k$ with $\Gamma'\subset \mbox{\rm Sing}\pi^*{\mathcal F}$. Moreover,  at a generic point $p\in \Gamma$ we have
a point $p'\in \Gamma'$ where $\rho_k$ induces an isomorphism
$$
(M,p')\rightarrow (M_k,p).
$$
In particular the pair ${\mathcal F}_k, E_k$ has a simple singularity at such points $p$.

We say that $\Gamma$ is a {\em trace curve}, respectively a {\em corner curve}, if and only if $\Gamma'$ is so.
If $\Gamma$ is a trace curve, there is exactly one partial separatrix $C$ such that $\Gamma'\subset C$. We say that $C$ is the {\em partial separatrix asociated to $\Gamma$} and we denote it by  $C_\Gamma$.

Let us note that $C=C_\Gamma$ if and only if $\Gamma\subset C_k$.

\begin{definition} Consider a point  $p\in F_k$ and a partial separatrix $C$ where $p\in C_k$. We say that $C$ is {\em complete at $p$} if
for any dicritical component $E_i$ of $E$
such that
$E_i\subset \rho_k^{-1}(p)$ we have $E_i\cap C = \emptyset$. Otherwise we say that $C$ is {\em incomplete at $p$}.
\end{definition}

\begin{remark}
\label{rk:sepparcialcompleta}
If $C$ is complete at $p$, we find a closed immersion
$$
(S_{C}, C\cap \rho_k^{-1}(p))\subset (M,\rho_k^{-1}(p) )
$$
where $(S_{C}, C\cap \rho_k^{-1}(p))$ is a finite union of germs of surface invariant for $\pi^*{\mathcal F}$. Taking the image by $\rho_k$, we obtain a finite union
$$
(\rho_k(S_C),p)\subset (M_k,p)
$$
of germs of surface at $p$ invariant for ${\mathcal F}_k$.
\end{remark}

\begin{remark} A partial separatrix $C$ is complete, as stated in the introduction, if and only if it is complete at the origin  $0\in {\mathbb C}^3$. On the other hand, any partial separatrix $C$ is complete at the points $p\in C\cap F$ in the final step of the reduction of singularities, even if $p$ belongs to a compact dicritical component $E_i$ of $E$.
\end{remark}

\begin{remark}
\label{rk:completitudequirreduction}
Let $p_k\in F_k$ be a point such that the center $Y_k$ of $\pi_{k+1}$ is a germ of curve with $p_k\in Y_k$. In view of the equireduction properties of the sequence  of reduction of singularities, we have that $\rho_k^{-1}(p_k)$ is a union of compact curves an hence it does not contain any component of $E$. Then any partial separatrix is complete at $p_k$.
\end{remark}

\begin{remark}
 \label{rk:completoestable}
 We have $C_k=\pi_{k+1}(C_{k+1})$. If the partial separatrix $C$ is complete at a point $p\in C_k\cap F_k$ then it is
 complete at all the points in $C_{k+1}\cap \pi_{k+1}^{-1}(p)$.
  Moreover, assume that $\pi_{k+1}$ satisfies one of the following conditions:
\begin{enumerate}
\item The center $Y_k$ of $\pi_{k+1}$ does not contain $p$.
\item The center $Y_k$ is a germ of curve.
\item The blow-up $\pi_k$ is non dicritical.
\end{enumerate}
Then the partial separatrix $C$ is complete at $p\in C_k\cap F_k$ if and only if it is complete at all the points $p'\in C_{k+1}\cap \pi_{k+1}^{-1}(p)$.
\end{remark}

\begin{proposition}
\label{pro:puntoscompletos}
Let $C$ be a partial separatrix complete at $p\in C_k\cap F_k$. We have:
\begin{enumerate}
\item[a)] If $\pi_{k+1}$ is centered at $p$, then $\pi_{k+1}$ is a non-dicritical blow-up.
\item[b)] If $p\in E^k_i$, where $E^k_i$ is compact invariant,  there is a compact trace curve $\Gamma\subset C_k\cap E^k_i$, with $p\in \Gamma$.
\item[c)] If $p\in E^k_j$, where $E^k_j$ is  compact dicritical, then   $C\cap E_j^N\ne \emptyset$.
\end{enumerate}
\end{proposition}
\begin{proof} We will do induction on $N-k$ to prove  statements a), b) and c) in this order. If $k=N$, we are done. Assume that $k<N$. Since it is a local problem at $p$, if $p\notin Y_k$ we conclude by induction. Thus, we  assume that $p\in Y_k$.

{\em First case: the center of $\pi_{k+1}$ is a germ of curve $(Y_k,p)$}. We only have  to prove b) and c). Note that for any compact component $E^{k}_s$ such that $p\in E^k_s$ we have $$\pi_{k+1}^{-1}(p)\subset E^{k+1}_s$$
and there is a point $p'\in C_{k+1}\cap \pi_{k+1}^{-1}(p)$. Assume b),
we know that $C$ is complete at $p'$ and by induction hypothesis on $p'\in E^{k+1}_i$ we conclude that
there is a trace compact curve $\Gamma'\subset E^{k+1}_i\cap C_{k+1}$ with $p'\in \Gamma'$. Now, it is enough to consider $\Gamma=\pi_{k+1}(\Gamma')$. Assume c),
we know that $C$ is complete at $p'$ and by induction hypothesis on $p'\in E^{k+1}_j$ we conclude that $C\cap E^N_j\ne \emptyset$.

{\em Second case: the center of $\pi_{k+1}$ is the point $p$}. We first prove a). If $\pi_{k+1}$ is a dicritical blow-up, there is a point $p'\in C_{k+1}\cap E^{k+1}_{k+1}$. We know that $C$ is complete at $p'$ and by induction hypotesis we apply c) at $p'$ to obtain that $C\cap E^N_{k+1}\ne\emptyset$.  This contradicts the fact that $C$ is complete at $p$.

Now we prove b) and c) already assuming that $\pi_{k+1}$ is non-dicritical. Take a point $p'\in C_{k+1}\cap E^{k+1}_{k+1}$, we know that $C$ is complete at $p'$. Since $E^{k+1}_{k+1}$ is compact and invariant, by induction hypothesis on $p'\in E^{k+1}_{k+1}$, we find a trace compact curve $\Gamma''\subset C_{k+1}\cap E^{k+1}_{k+1}$. Note that for any compact component $E^k_s$ such that $p\in E^k_s$ we have that
$$
E^{k+1}_s\cap E^{k+1}_{k+1}
$$
is a projective line in the projective plane $E^{k+1}_{k+1}$. In particular there is at least one point $p''_s\in \Gamma''\cap E^{k+1}_s$. We know that $C$ is complete at the point $p''_s$.   Assume b), we apply induction hypothesis on $p''_i\in E^{k+1}_{i}$ to find a trace compact curve $\Gamma'\subset E^{k+1}_i\cap C_{k+1}$ such that $p''_i\in \Gamma'$. We conclude by taking $\Gamma=\pi_{k+1}(\Gamma')$. Assume c), we apply induction hypothesis on $p''_j\in E^{k+1}_{j}$ to find that $C\cap E^{N}_j\ne \emptyset$.
\end{proof}

\begin{remark} Proposition \ref{pro:puntoscompletos} can also be proved by invoking the germ of surface $(S_{C},p)$ obtained in
Remark \ref{rk:sepparcialcompleta}
and considering the intersections with the corresponding compact component of $E^k$. We have used the inductive arguments  because of the general style of the paper.
\end{remark}

\begin{proposition}
\label{pro:incompletodicriticoono}
 Let $C$ be an incomplete partial separatrix and consider an index  $0\leq k\leq N$. Then there is a point
 $p\in C_k$ such that $C$ is not complete at $p$ or there is a compact dicritical component $E^k_j$ such that $C_k\cap E_j^k\not=\emptyset$.
\end{proposition}
\begin{proof} Induction on $N-k$. If $k=N$ we are done, since $C$ intersects at least one compact dicritical component of the exceptional divisor. Take $k<N$. In order to find a contradiction, assume
that $C$ is complete at any $p\in C_k$ and that it does not intersect any compact dicritical component in $E^k$. We already know that $C$ is complete at any point in $C_{k+1}$ by Remark \ref{rk:completoestable}. By Proposition \ref{pro:puntoscompletos}, we have that if  $E^{k+1}_{k+1}$ is a compact component with $E^{k+1}_{k+1}\cap C_{k+1}\ne\emptyset$, then $E^{k+1}_{k+1}$ is an invariant component. This gives the desired contradiction by applying induction hypothesis.
\end{proof}

\begin{definition} We say that $p\in F_k$ is an {\em incomplete point} if and only if there is a partial separatrix $C$ such that $p\in C_k$ and $C$ is incomplete at $p$.
\end{definition}
If there are no partial separatrices $C$ such that $p\in C_k$ the point $p$ is considered to be  complete.

\section{Transition of Camacho-Sad indices}
Let us consider an irreducible compact curve $\Gamma\subset F_k\cap \mbox{\rm Sing}{\mathcal F}_k$ and an invariant compact component $E^k_i$ of $E^k$ such that $\Gamma\subset E^k_i$. Note that, since $\Gamma$ is compact, there are no dicritical components containing $\Gamma$.

Consider a plane  section $\Delta$ transverse to $\Gamma$ at a generic point $p\in \Gamma$. Taking appropriate local coordinates $x,y$ at $p\in \Delta$, the restricted foliation ${\mathcal F}_k\vert_\Delta$ is given by a $1$-form
$$
\omega=x\left\{(\lambda x+\mu y+\phi(x,y))\frac{dx}{x}-dy\right\}
$$
where $E^k_i\cap \Delta=(x=0)$, $\mu\ne 0$ and $\phi(x,y)$ has a zero of order at least two at the origin. The Camacho-Sad index of ${\mathcal F}_k\vert_\Delta$ at $p$ with respect to the invariant curve $x=0$ is by definition the value $1/\mu$, see \cite{Cam-S, Can-C-D}. We denote
$$
\mbox{\rm Ind}({\mathcal F},E_i;\Gamma)= \mbox{\rm Ind}({\mathcal F}_k\vert_\Delta, E^k_i\cap \Delta;p)=1/\mu.
$$
This index may be calculated in any step $k'\geq k$ of the reduction of singularities and at any point of the strict transform of $\Gamma$ of dimensional type two.

\begin{remark}
\label{rk:productoindices}
Assume that $\Gamma$ is contained in two compact invariant components $E^k_i, E^k_j$ of $E^k$. We have that
$\Gamma=E^k_i\cap E^k_j$. By the general properties of Camacho-Sad index \cite{Cam-S}, we have that
$$
\mbox{Ind}({\mathcal F},E_i;\Gamma)\mbox{Ind}({\mathcal F},E_j;\Gamma)=1.
$$
\end{remark}
\begin{definition}
\label{def:nodalsaddle}
 Let $\Gamma\subset \mbox{\rm Sing}(\pi^*{\mathcal F})$ be a compact curve contained in a compact invariant component $E_i$.
\begin{enumerate}
\item $\Gamma$ is a {\em nodal curve} if and only if $\mbox{Ind}({\mathcal F},E_i;\Gamma)\in {\mathbb R}_{>0}\setminus {\mathbb Q}$.
\item $\Gamma$ is a {\em real saddle curve} if and only if $\mbox{Ind}({\mathcal F},E_i;\Gamma)\in {\mathbb R}_{<0}$.
\item $\Gamma$ is a {\em complex saddle curve} if and only if $\mbox{Ind}({\mathcal F},E_i;\Gamma)\in {\mathbb C}\setminus {\mathbb R}$.
\end{enumerate}
\end{definition}
 Note that by Remark \ref{rk:productoindices} the definition above does not depend on the invariant component $E_i$ of $E$ such that $\Gamma\subset E_i$.

Take a point $p\in E^k_i$ where $E^k_i$ is a compact invariant component of $E^k$. We are interested in considering irreducible germs of curves
$$(\gamma,p)\subset (\mbox{\rm Sing}{\mathcal F}_k\cap E^k_i,p)
$$
Such a germ $(\gamma,p)$ is contained in exactly one compact curve $\Gamma\subset \mbox{\rm Sing}{\mathcal F}_k\cap E^k_i$. This allows us to put
$$
\mbox{Ind}({\mathcal F},E_i;\gamma)=\mbox{Ind}({\mathcal F},E_i;\Gamma).
$$
We denote ${\mathcal B}^k_i(p)$ the set of irreducible germs of curves $(\gamma,p)\subset (\mbox{\rm Sing}{\mathcal F}_k\cap E^k_i,p)$.

In  Proposition \ref{pro:indicesymultiplicidad} we precise a relationship between the indices, counted with multiplicity,  with respect to two incident compact components.
\begin{proposition}
\label{pro:indicesymultiplicidad}
 Consider a point $p\in \Gamma=E^k_i\cap E^k_j$ where $E^k_i$ and $E^k_j$ are compact components of $E^k$. Assume that $E^k_i$ is an invariant component of $E^k$.
\begin{enumerate}
\item[a)] If $E^k_j$ is a dicritical component and ${\mathcal G}={\mathcal F}_k\vert_{E^k_j}$ we have
\begin{equation*}
\mbox{\rm Ind}({\mathcal G},\Gamma;p)=
\sum_{\gamma\in {\mathcal B}^k_i(p)}(\gamma,\Gamma)_p \mbox{\rm Ind}({\mathcal F},E_i;\gamma),
\end{equation*}
where $(\gamma,\Gamma)_p$ is the intersection multiplicity of $\gamma,\Gamma$ at $p$.
\item[b)] If $E^k_j$ is an invariant component and $\alpha=\mbox{\rm Ind}({\mathcal F},E_i;\Gamma)$, we have
\begin{equation*}
\sum_{\gamma\in {\mathcal B}^k_i(p)\setminus \{\Gamma\}}(\gamma,\Gamma)_p \mbox{\rm Ind}({\mathcal F},E_i;\gamma)=-\alpha
\sum_{\delta\in {\mathcal B}^k_j(p)\setminus\{\Gamma\}}(\delta,\Gamma)_p \mbox{\rm Ind}({\mathcal F},E_j;\delta).
\end{equation*}
\end{enumerate}
\end{proposition}

\begin{proof} We do induction on $N-k$. Let us consider first the case $k=N$:
\begin{enumerate}

\item If $p$ is non singular, it belongs to at most one invariant component of the divisor. We have a) with ${\mathcal B}^k_i(p)=\emptyset$. Thus we are done.

\item  Assume that $p$ is of dimensional type two and $E_j$ is a dicritical component. Then the singular locus is non singular at $p$ and  $E_j$ gives a section transversal to it. We are done by the definition of the generic index.

\item Assume $p$ is of dimensional type two and $E_j$ is invariant. The singular locus is $\Gamma$. In this case ${\mathcal B}^k_i(p)={\mathcal B}^k_j(p)=\{\Gamma\}$ and there is nothing to prove.

\item
 \label{enumerate4}
 Assume that $p$ is of dimensional type three. Then $E_j$ is necessarily invariant  and  there are local coordinates $x,y,z$ at $p$ such that
$$
E_i=(x=0),\; E_j=(y=0),\quad  \Gamma=(x=y=0),
$$
the plane $z=0$ is invariant and the singular locus is given by $\Gamma\cup \gamma\cup \delta$, where
 $$ \gamma=(x=z=0);\quad \delta= (y=z=0).$$
Moreover, the foliation ${\mathcal F}$ is given locally at $p$ by an integrable 1-form of the type
$$
\omega=\frac{dx}{x}+ (-\alpha+b(x,y,z))\frac{dy}{y}+(-\beta+c(x,y,z))\frac{dz}{z},\quad \alpha\ne 0\ne\beta.
$$
By the integrability condition $\omega\wedge dw=0$, we have
\begin{eqnarray*}
b(x,y,z)&=&xb'(x,y,z)+y b''(x,y,z) \\ c(x,y,z)&=& xc'(x,y,z)+yc''(x,y,z)
\end{eqnarray*}
and thus
$$
\frac{-\beta+c(x,y,z)}{-\alpha+b(x,y,z)}=\frac{\beta}{\alpha}+yf'(x,y,z)+zf''(x,y,z).
$$
Then, we have
\begin{eqnarray*}
\mbox{Ind}({\mathcal F},E_i;\Gamma)=\alpha,\;
 \mbox{Ind}({\mathcal F},E_i;\gamma)=\beta,\;
 \mbox{Ind}({\mathcal F},E_j;\delta)=-\beta/\alpha.
\end{eqnarray*}
The desired relation is
$
\beta= -\alpha (-\beta/\alpha)
$,
that is obviously satisfied.
\end{enumerate}

Now, suppose that $k<N$. If $p\notin Y_k$ we are done by induction; hence we  assume $p\in Y_k$. Moreover, the center $Y_k$ of the blow-up $\pi_{k+1}$ cannot be a germ of curve, since there are two compact components of $E^k$ through $p$ and $Y_k$ should have normal crossings with $E^k$. Thus $Y_k=\{p\}$.

Let us give some remarks and fix notations.
 We put
$$
\Gamma'=E^{k+1}_i\cap E^{k+1}_j,\; L'_i=E^{k+1}_i\cap E^{k+1}_{k+1},\; L'_j=E^{k+1}_j\cap E^{k+1}_{k+1},\; p'=\Gamma'\cap L'_i.
$$
 In view of Noether's formula for the intersection multiplicity (see \cite{Ful} for instance), given \break $\gamma\in {\mathcal  B}(i;p)\setminus \{\Gamma\}$ we have
\begin{equation*}
 (\gamma,\Gamma)_p=(\gamma',\Gamma')_{p'}+\sum_{q\in L'_i} (\gamma',L'_i)_q
 \end{equation*}
 where $\gamma'$ stands for the strict transform of $\gamma$.
Let us also note that
$$
\mbox{Ind}({\mathcal F},E_i;\gamma)= \mbox{Ind}({\mathcal F},E_i;\gamma'),
$$
since the computations are made at generic points of $\gamma\subset E_i^k$.

On the other hand, note that $E^{k+1}_{k+1}$ is a projective plane and $L'_i,L'_j\subset E^{k+1}_{k+1}$ are both projective lines. In particular, let $\Lambda\subset E^{k+1}_{k+1}$ be a global irreducible curve $\Lambda\subset \mbox{\rm Sing}({\mathcal F}_{k+1})$ of degree $d_\Lambda$ with $\Lambda\ne L'_i,L'_j$. By Bezout's Theorem, we know that
\begin{equation*}
d_\Lambda=\sum_{q\in L'_i;\;q\in \delta\subset \Lambda} (\delta,L'_i)_q= \sum_{q\in L'_j;\;q\in \delta\subset \Lambda} (\delta,L'_j)_q ,
\end{equation*}
where $\delta$ runs over the irreducible branches of $\Lambda$ at $q$.

Now, we have four cases to consider:
\begin{enumerate}
\item[i)] $E^k_j$ is dicritical and $\pi_{k+1}$ is a dicritical blow-up.
\item[ii)] $E^k_j$ is dicritical and $\pi_{k+1}$ is a non dicritical blow-up.
\item[iii)] $E^k_j$ is invariant and $\pi_{k+1}$ is a dicritical blow-up.
\item[iv)] $E^k_j$ is invariant and $\pi_{k+1}$ is a non dicritical blow-up.
\end{enumerate}
 
{\em Assume first that $E^k_j$ is a dicritical component.} Let us note that
$$\Gamma\not\subset\mbox{\rm Sing}({\mathcal F}_k);\;\Gamma',L'_j\not\subset\mbox{\rm Sing}({\mathcal F}_{k+1}).$$
The induced  induced foliation  ${\mathcal G}'$ by ${\mathcal F}_{k+1}$ on $E^{k+1}_j$ is the transform of $\mathcal G$ by the restriction
 $$
 \tilde\pi_{k+1}: E^{k+1}_j\rightarrow E^k_j
 $$
 of the blow-up $\pi_{k+1}$. In particular, by the known properties of Camacho-Sad index (see \cite{Cam-S, Can-C-D}) we have that
 \begin{equation*}
 \label{eq:csindex}
 \mbox{\rm Ind}({\mathcal G},\Gamma;p)= \mbox{\rm Ind}({\mathcal G}',\Gamma';p')+1.
 \end{equation*}

{\em First case: $\pi_{k+1}$ is a dicritical blow-up}. Let us denote  ${\mathcal G}_1$ the induced foliation by ${\mathcal F}_{k+1}$ on $E^{k+1}_{k+1}$. The self-intersection of the projective line $L'_i$ in the projective plane $E^{k+1}_{k+1}$ is equal to $+1$. Then  we have
\begin{equation*}
\sum_{q\in L'_i}\mbox{\rm Ind}({\mathcal G}_1,L'_i;q)=+1,
\end{equation*}
Let us note that since $\Gamma', L'_i$ are not in the singular locus we have a bijection
$$
{\mathcal B}^k_i(p)\leftrightarrow \bigcup_{q\in L'_i} {\mathcal B}^{k+1}_i(q)
$$
given by the strict transform $\gamma\mapsto \gamma'$. Applying induction hypothesis to the points of $L'_i$   we deduce that
$$
\sum_{q\in L'_i;\gamma\in {\mathcal B}^k_i(p)}(\gamma',L'_i)_q\mbox{\rm Ind}({\mathcal F},E_i;\gamma)=\sum_{q\in L'_i}\mbox{\rm Ind}({\mathcal G}_1,L'_i;q)=+1.
$$
Applying induction hypothesis at $p'$ as well, we have
\begin{eqnarray*}
&&\sum_{\gamma\in {\mathcal B}^k_i(p)}(\gamma,\Gamma)_p \mbox{Ind}({\mathcal F},E_i;\gamma)=\\
&&=\sum_{\gamma\in {\mathcal B}^k_i(p)}
\left(
(\gamma',\Gamma')_{p'}
+\sum_{q\in L'_i}(\gamma',L'_i)_q \right)
\mbox{Ind}({\mathcal F},E_i;\gamma)=\\
&&=\mbox{Ind}({\mathcal G}',\Gamma';p')+1= \mbox{Ind}({\mathcal G},\Gamma;p).
\end{eqnarray*}
This case is ended.

{\em Second case: $\pi_{k+1}$ is a non dicritical blow-up}.
Let us denote $\tilde \alpha= \mbox{Ind}({\mathcal F},E_{i};L'_i)$. By induction hypothesis at $p'$ we have
\begin{equation*}
\mbox{Ind}({\mathcal G}',\Gamma';p')=\tilde \alpha +
\sum_{\gamma\in {\mathcal B}^k_i(p)}(\gamma',\Gamma')_{p'} \mbox{Ind}({\mathcal F},E_i;\gamma'),
\end{equation*}
and since $\mbox{Ind}({\mathcal G}',\Gamma';p')=\mbox{Ind}({\mathcal G},\Gamma;p)-1$, we can put
\begin{equation*}
\mbox{Ind}({\mathcal G},\Gamma;p)=\tilde \alpha +1+
\sum_{\gamma\in {\mathcal B}^k_i(p)}(\gamma',\Gamma')_{p'} \mbox{Ind}({\mathcal F},E_i;\gamma').
\end{equation*}
By Noether's Theorem we have
$
(\gamma',\Gamma')_{p'}=(\gamma,\Gamma)_p-\sum_{q\in L'_i}(\gamma',L'_i)_q
$.
Then
\begin{equation*}
\mbox{Ind}({\mathcal G},\Gamma;p)= \sum_{\gamma\in {\mathcal B}^k_i(p)}(\gamma,\Gamma)_p \mbox{Ind}({\mathcal F},E_i;\gamma)
+ \beta
\end{equation*}
where
$$
\beta= \tilde \alpha +1-
\sum_{q\in L'_i}\sum_{\gamma\in {\mathcal B}^k_i(p)}(\gamma',L'_i)_q \mbox{Ind}({\mathcal F},E_i;\gamma').
$$
Now, it is enough to show that $\beta=0$. By induction hypothesis in the statement b) referred to
$E^{k+1}_i$ and $E^{k+1}_{k+1}$ we have that
\begin{equation*}
\sum_{q\in L'_i;\;\gamma\in {\mathcal B}^k_i(p)}(\gamma',L'_i)_q \mbox{Ind}({\mathcal F},E_i;\gamma')=-\tilde\alpha
\sum_{\Lambda\subset E^{k+1}_{k+1},\;\Lambda\ne L'_i}d_\Lambda \mbox{Ind}({\mathcal F},E_{k+1};\Lambda).
\end{equation*}
Now, we apply induction hypothesis in the statement a) referred to $E^{k+1}_{k+1}$ and $E^{k+1}_j$ to obtain
\begin{equation*}
\sum_{q\in L'_j}\mbox{Ind}({\mathcal G}',L'_j;q)=1/\tilde \alpha+
\sum_{\Lambda\subset E^{k+1}_{k+1};\Lambda\ne L'_i}d_{\Lambda}\mbox{Ind}({\mathcal F},E_{k+1};\Lambda),
\end{equation*}
where $1/\tilde\alpha= \mbox{Ind}({\mathcal F},E_{k+1};L'_i)$. Recalling that the self intersection of $L'_j$ in $E^{k+1}_j$ is equal to $-1$, we have
$
-1= \sum_{q\in L'_j}\mbox{Ind}({\mathcal G}',L'_j;q)
$ and
we obtain
$$
-1=1/\tilde\alpha -(1/\tilde\alpha)\sum_{q\in L'_i;\;\gamma\in {\mathcal B}^k_i(p)}(\gamma',L'_i)_q \mbox{Ind}({\mathcal F},E_i;\gamma').
$$
That is
\begin{eqnarray*}
\beta=\tilde \alpha +1- \sum_{q\in L'_i;\;\gamma\in {\mathcal B}^k_i(p)}(\gamma',L'_i)_q \mbox{Ind}({\mathcal F},E_i;\gamma')=0
\end{eqnarray*}
and we are done.

{\em Let us suppose finally that  $E^k_j$ is an invariant component.}

{\em First case: $\pi_{k+1}$ is a  dicritical blow-up}. Let ${\mathcal G}_1$ be the induced foliation by ${\mathcal F}_{k+1}$ on $E^{k+1}_{k+1}$.
By applying induction hipothesis at $L'_i$ and $L'_j$  and recalling that the self-intersection of $L'_i,L'_j$ inside $E^{k+1}_{k+1}$ is $+1$, we have
$$
1= \sum_{q\in L'_i}\mbox{Ind}({\mathcal G}_1,L'_i;q)= \sum_{q\in L'_j}\mbox{Ind}({\mathcal G}_1,L'_j;q),
$$
and hence
\begin{eqnarray}
\label{eq:csindexdos}
1=\alpha+
\sum_{q\in L'_i}\sum_{\gamma\in {\mathcal B}^k_i(p)\setminus \{\Gamma\}}(\gamma',L'_i)_q\mbox{Ind}({\mathcal F},E_i;\gamma'),\\
\label{eq:csindextres}
1= (1/\alpha)+
\sum_{q\in L'_j}\sum_{\delta\in {\mathcal B}^k_j(p)\setminus \{\Gamma\}}(\delta',L'_j)_q\mbox{Ind}({\mathcal F},E_j;\delta').
\end{eqnarray}
Also, by induction hypothesis at $p'$ referred to $E^{k+1}_i$ and $E^{k+1}_j$ we have
\begin{equation}
\label{eq:csindexcuatro}
\sum_{\gamma\in {\mathcal B}^k_i(p)\setminus\{\Gamma\}}(\gamma',\Gamma')_{p'} \mbox{Ind}({\mathcal F},E_i;\gamma')=-\alpha
\sum_{\delta\in {\mathcal B}^k_j(p)\setminus\{\Gamma\}}(\delta',\Gamma')_{p'} \mbox{Ind}({\mathcal F},E_j;\delta').
\end{equation}
 Using Noether's formula and the equalities (\ref{eq:csindexdos}, \ref{eq:csindextres}) , we have
 \begin{eqnarray*}
 \sum_{\gamma\in {\mathcal B}^k_i(p)\setminus\{\Gamma\}}(\gamma,\Gamma)_{p} \mbox{Ind}({\mathcal F},E_i;\gamma)= \sum_{\gamma\in {\mathcal B}^k_i(p)\setminus\{\Gamma\}}(\gamma',\Gamma')_{p'} \mbox{Ind}({\mathcal F},E_i;\gamma')+ (1-\alpha),\\
  -\alpha\sum_{\delta\in {\mathcal B}^k_j(p)\setminus\{\Gamma\}}(\delta,\Gamma)_{p} \mbox{Ind}({\mathcal F},E_i;\delta)= -\alpha\sum_{\gamma\in {\mathcal B}^k_j(p)\setminus\{\Gamma\}}(\delta',\Gamma')_{p'} \mbox{Ind}({\mathcal F},E_i;\delta')+ (1-\alpha)
 \end{eqnarray*}
 and we are done by Equation (\ref{eq:csindexcuatro}).

 {\em Second case: $\pi_{k+1}$ is a non dicritical blow-up}. Let us denote
 $$
 \beta=\mbox{Ind}({\mathcal F}, E_i;L'_i);\quad \rho= \mbox{Ind}({\mathcal F}, E_j;L'_j).
 $$
 We have
 $
 1/\beta=\mbox{Ind}({\mathcal F}, E_{k+1};L'_i)$ and $1/\rho= \mbox{Ind}({\mathcal F}, E_{k+1};L'_j).
 $
 Let us put
 $$
 \epsilon= \sum_{\Lambda\subset E_{k+1}^{k+1}, \Lambda\ne L'_i,L'_j}d_\Lambda \mbox{Ind}({\mathcal F}, E_{k+1};\Lambda).
 $$
 Now, if we take a generic plane section $\Delta$ at $p$ and we apply Camacho-Sad's equality to ${\mathcal F}_k\vert_\Delta$ after the blow-up $\pi_{k+1}$, we obtain
 \begin{equation*}
 \label{eq:cssection}
 -1= 1/\beta +1/\rho +\epsilon.
 \end{equation*}
 By induction hypothesis referred to $E_i^{k+1}$ and $E_{k+1}^{k+1}$, we have the following equality
 \begin{eqnarray*}
 \alpha+\sum_{q\in L'_i}\sum_{\gamma\in {\mathcal B}^k_i(p)\setminus\{\Gamma\}}(\gamma',L'_i)_q \mbox{Ind}({\mathcal F}, E_i;\gamma')=
 -\beta((1/\rho)+\epsilon)=\beta +1
 \end{eqnarray*}
 and thus
 \begin{equation*}
 \sum_{q\in L'_i}\sum_{\gamma\in {\mathcal B}^k_i(p)\setminus\{\Gamma\}}(\gamma',L'_i)_q \mbox{Ind}({\mathcal F}, E_i;\gamma')=-\alpha +\beta +1.
 \end{equation*}
 Now, applying induction referred to $E_j^{k+1}$ and $E_{k+1}^{k+1}$, we have
 \begin{eqnarray*}
 (1/\alpha)+\sum_{q\in L'_i}\sum_{\delta\in {\mathcal B}^k_j(p)\setminus\{\Gamma\}}(\delta',L'_i)_q \mbox{Ind}({\mathcal F}, E_i;\delta')=
 -\rho((1/\beta)+\epsilon)=\rho+1
 \end{eqnarray*}
 and thus
 \begin{equation*}
 -\alpha \sum_{q\in L'_i}\sum_{\delta\in {\mathcal B}^k_j(p)\setminus\{\Gamma\}}(\delta',L'_i)_q \mbox{Ind}({\mathcal F}, E_i;\delta')=1-\alpha(\rho+1).
 \end{equation*}
 Applying induction hypothesis at $p'$, we have
 \begin{eqnarray*}\beta+
 \sum_{\gamma\in {\mathcal B}^k_i(p)\setminus\{\Gamma\}}(\gamma',\Gamma')_{p'} \mbox{Ind}({\mathcal F}, E_j;\gamma')=\\=
 -\alpha\left(\rho +\sum_{\delta\in {\mathcal B}^k_j(p)\setminus\{\Gamma\}}(\delta',\Gamma')_{p'} \mbox{Ind}({\mathcal F}, E_j;\delta')\right).
 \end{eqnarray*}
 Thus, by Noether's equality, we only have to verify that
 $$
 (-\alpha+\beta+1) -\beta= (1-\alpha(\rho+1)) +\alpha\rho
 $$
and this is evident.
 \end{proof}

\begin{corollary}
\label{cor:indicesymultiplicidad}
Let $p\in F_k$ be a point such that $p\in E^k_i\cap E^k_j\cap E^k_\ell$ where $E^k_i$, $E^k_j$ and $E^k_\ell$ are compact invariant components of $E^k$. Let us denote
\begin{eqnarray*}
&\Gamma_\ell= E^k_i\cap E^k_j, \; \Gamma_j= E^k_i\cap E^k_\ell,\; \Gamma_i= E^k_j\cap E^k_\ell.\\
&\alpha=\mbox{\rm Ind}({\mathcal F}, E_i;\Gamma_\ell),\;\beta=\mbox{\rm Ind}({\mathcal F}, E_i;\Gamma_j), \;\rho=
\mbox{\rm Ind}({\mathcal F}, E_j;\Gamma_i).
\end{eqnarray*}
Then, we have
$
\beta=-\alpha\rho
$.
\end{corollary}
\begin{proof} Induction on $N-k$. If $k=N$, we are done (see (\ref{enumerate4}) in the proof of the case $k=N$ in Proposition \ref{pro:indicesymultiplicidad}).
Assume that $k<N$ and $p\in Y_k$ as in previous proofs. We have that $Y_k=\{p\}$.  If the blow-up is non dicritical, we put
$$
\nu=\mbox{\rm Ind}({\mathcal F}, E_i;E_i\cap E_{k+1}),\; \xi=\mbox{\rm Ind}({\mathcal F}, E_\ell;E_\ell\cap E_{k+1}),\; \mu=\mbox{\rm Ind}({\mathcal F}, E_j;E_j\cap E_{k+1}).
$$
By induction hypothesis, we have
$
\nu=-\alpha\mu, \;\nu=-\xi\beta,\; \mu=-\xi\rho
$.
That is, we have $\xi\beta=\alpha\mu=-\alpha\xi\rho$ and thus $\beta=-\alpha\rho$.

Assume now that the blow-up is dicritical. We denote by $\Gamma'_\ell,\Gamma'_j,\Gamma'_i$ the strict transforms of $\Gamma_\ell,\Gamma_j,\Gamma_i$ respectively. We also denote by
\begin{equation*}
{\mathcal B}_i^*={\mathcal B}^k_i(p)\setminus \{\Gamma_\ell,\Gamma_j\};\quad
{\mathcal B}_j^*={\mathcal B}^k_j(p)\setminus \{\Gamma_\ell,\Gamma_i\};\quad
{\mathcal B}_\ell^*={\mathcal B}^k_\ell(p)\setminus \{\Gamma_i,\Gamma_j\};
\end{equation*}
and we put
$$
I_u^{(v)}=\sum_{\gamma\in {\mathcal B}_u^*}(\Gamma_v,\gamma)_p\mbox{\rm Ind}({\mathcal F},E_u;\gamma)
$$
for  $u\ne v$ with $u,v\in\{i,j,\ell\}$.
Given a germ of curve $\gamma$ we denote by $\gamma'$ the strict transform of $\gamma$,  as usual. Take also the following notations
$$
p'_u=\Gamma'_u\cap E_{k+1}^{k+1},\;  L'_{u}= E_u^{k+1}\cap E^{k+1}_{k+1}; \quad u\in \{i,j,\ell\}
$$
and
\begin{equation*}
I'_u=\sum_{{q'\in L'_u}, {\gamma\in {\mathcal B}_u^*}}(L'_u,\gamma')_{q'}\mbox{\rm Ind}({\mathcal F},E_i;\gamma)
;\quad u\in \{i,j,\ell\}.
\end{equation*}
Finally, we put
$$
{I'}_u^{(v)}=\sum_{\gamma\in {\mathcal B}_u^*}(\Gamma'_v,\gamma')_{p'_v}\mbox{\rm Ind}({\mathcal F},E_u;\gamma')
$$
for  $u\ne v$ with $u,v\in\{i,j,\ell\}$.

By Noether's formula, we have
$$
I_u^{(v)}=I'_u+{I'}_i^{(v)}, \mbox{ for } u,v\in\{i,j,\ell\},  u\ne v.
$$
Now, by applying  part a) of Proposition \ref{pro:indicesymultiplicidad} to the exceptional divisor and each of three other divisors, we have
\begin{eqnarray}
\label{eq:dicritica}
\alpha+\beta+{I'}_i=\frac{1}{\beta}+\frac{1}{\rho}+I'_\ell=\frac{1}{\alpha}+\rho+I'_{j}=1.
\end{eqnarray}
By Proposition \ref{pro:indicesymultiplicidad} we have
\begin{eqnarray*}
&I_i^{(\ell)}=-\alpha I_j^{(\ell)}; \; I_i^{(j)}=-\beta I_\ell^{(j)};\; I_\ell^{(i)}=-(1/\rho) I_j^{(i)}.\\
&{I'}_i^{(\ell)}=-\alpha {I'}_j^{(\ell)}; \; {I'}_i^{(j)}=-\beta {I'}_\ell^{(j)};\; {I'}_\ell^{(i)}=-(1/\rho) {I'}_j^{(i)}.
\end{eqnarray*}
We deduce that
\begin{eqnarray*}
&I'_i=-\alpha I'_j; \; I'_i=-\beta I'_\ell;\; I'_\ell=-(1/\rho) I'_j.
\end{eqnarray*}
This implies that
$$
-\alpha I'_j=\beta(1/\rho) I'_j.
$$
Hence, if $I'_j\ne 0$ we conclude that $\beta=-\alpha\rho$ and we are done.

Assume now that $I'_j=0$ and hence $I'_i=I'_j=I'_\ell=0$.  By Equation \ref{eq:dicritica} we have
$$
\alpha+\beta=(1/\beta)+(1/\rho)=(1/\alpha)+\rho=1.
$$
We deduce that $1+\alpha\rho=\alpha$, hence $1=\alpha(1-\rho)$, but $1=\alpha+\beta$. This implies that $\beta=-\alpha\rho$ as desired.
\end{proof}
Let us consider a point $p$ in a compact invariant component $E^k_i$ of $E^k$ and a partial separatrix $C$. We denote by
$
{\mathcal B}^k_i(C;p)
$
the set of germs of curve $(\gamma,p)$ such that $(\gamma,p)\subset (C_k\cap E^k_i,p)$.
\begin{corollary}
\label{cor:transitioncompleta}
Let us consider a point  $p\in \Gamma= E^k_i\cap E^k_j$, where $E^k_i$ and $E^k_j$ are compact invariant components of $E^k$ and a partial separatrix $C$. Assume that $p$ is a complete point for $C$.
We have
$$
\sum_{\gamma\in{\mathcal B}_i^k(C;p)}(\gamma,\Gamma)_p\mbox{\rm Ind}({\mathcal F},E^k_i;\gamma)=-\alpha \sum_{\eta\in{\mathcal B}_j^k(C;p)}(\eta,\Gamma)_p\mbox{\rm Ind}({\mathcal F},E^k_j;\eta).
$$
where $\alpha=\mbox{\rm Ind}({\mathcal F},E^k_i;\Gamma)$.
\end{corollary}
\begin{proof} We do induction on $N-k$ as usual. If $k=N$ we are done, by the local expression at simple points. Assume that $k<N$.   We suppose without loss of generality that $p\in Y_k$  and thus the next blow-up  $\pi_{k+1}$ is centered at the point $p$. Since $C$ is complete at $p$, the blow-up is non-dicritical. Put

$$
\Gamma'=E^{k+1}_{i}\cap E^{k+1}_j,\;  L'_i=E^{k+1}_{k+1}\cap E^{k+1}_i, \; L'_j=E^{k+1}_{k+1}\cap E^{k+1}_j
$$ and let $p'=\Gamma'\cap E^{k+1}_{k+1}$. Denote as usual by $\gamma'$ the strict transform of the germ of curve $\gamma$ and put
\begin{eqnarray*}
&I'_u(C)=\sum_{q\in L'_u}\sum_{\gamma\in {\mathcal B}_u^k(C;p)}(\gamma',L'_u)_q \mbox{\rm Ind}({\mathcal F},E^k_u;\gamma),\\
&I''_u(C)= \sum_{\gamma\in {\mathcal B}_u^k(C;p)}(\gamma',\Gamma')_{p'} \mbox{\rm Ind}({\mathcal F},E^k_u;\gamma)
\end{eqnarray*}
for $u\in\{i,j\}$.
We have that
$$
\sum_{\gamma\in{\mathcal B}_u^k(C;p)}(\gamma,\Gamma)_p\mbox{\rm Ind}({\mathcal F},E^k_u;\gamma)=I'_u(C)+I''_u(C),\;u\in\{i,j\}
$$
and by induction hypothesis we know that $I''_i(C)=-\alpha I''_j(C)$. Let us denote
$$
\beta=\mbox{\rm Ind}({\mathcal F},E^{k+1}_i;L'_i);\quad  \rho=\mbox{\rm Ind}({\mathcal F},E^{k+1}_j;L'_j).
$$
Also by induction hypothesis, we have
$$
(-1/\beta)I'_i(C)=\sum_{\Lambda}d_\Lambda \mbox{\rm Ind}({\mathcal F},E^{k+1}_{k+1};\Lambda)= (-1/\rho)I'_j(C),
$$
where $\Lambda$ stands for the global irreducible curves $\Lambda \subset E^{k+1}_{k+1}$ such that $\Lambda\subset C_{k+1}$. Applying Corollary \ref{cor:indicesymultiplicidad}, we deduce that
$$
I'_i(C)=({\beta}/{\rho}) I'_j(C)=-\alpha I'_j(C)
$$
and we are done.
\end{proof}

\section{Indices of partial separatrices}
Consider a partial separatrix $C$. Here we show that it is possible to define the index $\mbox{\rm Ind}(C;E_i)$ relative to any invariant compact component $E_i$ of $E$. Given an invariant compact component $E_i$ of $E$, we denote by
$
{\mathcal B}_iC
$ the set of global irreducible curves in $C\cap E_i$.  We put $\mbox{\rm Ind}(C;E_i)=0$ if there is no compact curve of $C$ contained in $E_i$. Otherwise, we shall put
$$
\mbox{\rm Ind}(C;E_i)=\mbox{\rm Ind}({\mathcal F}, E_i;\Gamma)
$$
where $\Gamma\in {\mathcal B}_iC$. Proposition \ref{pro:indexpartialseparatriz} assures that the definition is consistent.
\begin{proposition}
 \label{pro:indexpartialseparatriz}
 Let $C$ be a partial separatrix and $E_i$ a compact invariant component of $E$. If $\Gamma_1,\Gamma_2\in {\mathcal B}_iC$, we have
$
\mbox{\rm Ind}({\mathcal F}, E_i;\Gamma_1)=\mbox{\rm Ind}({\mathcal F}, E_i;\Gamma_2)
$.
\end{proposition}
Before giving the proof of Proposition \ref{pro:indexpartialseparatriz}, let us introduce the {\em dual graph ${\mathcal G}_N$ of the compact invariant components}. This graph has vertices corresponding to the compact invariant components;  two vertices $E_i, E_j$ are joined by a wedge if and only if $E_i\cap E_j\ne \emptyset$. It is the last one of the series of dual graphs ${\mathcal G}_k$ of the compact invariant components of $E^k$.

Since each new invariant compact component is produced by the blow-up of a point, we see that given two compact invariant components $E_i$ and $E_j$ we have that either $E_i\cap E_j=\emptyset$ or $E_i\cap E_j$ is an irreducible compact curve.

The graph ${\mathcal G}_{k+1}$ is obtained from ${\mathcal G}_k$ as follows. If the blow-up $\pi_{k+1}$ is dicritical, then ${\mathcal G}_{k+1}={\mathcal G}_k$. If the center of $\pi_{k+1}$ is a curve, we also have that ${\mathcal G}_{k+1}={\mathcal G}_k$. If $\pi_{k+1}$ is non dicritical and the center is a point $p_k$, we have four possibilities:
\begin{enumerate}
\item The point $p_k$ does not belong to any invariant compact component of $E^k$. In this case, the graph ${\mathcal G}_{k+1}$ is obtained from ${\mathcal G}_k$ by adding a new connected component to ${\mathcal G}_k$ consisting in a single vertex that represents the exceptional divisor $E^{k+1}_{k+1}$. No new wedges are added.
\item The point $p_k$ belongs to a single invariant compact component $E^k_i$ of $E^k$. Then ${\mathcal G}_{k+1}$ is obtained from ${\mathcal G}_k$ by adding a new vertex that represents the exceptional divisor $E^{k+1}_{k+1}$ and a new wedge connecting it with $E^{k+1}_i$.
\item The point $p_k$ belongs to exactly two invariant compact components $E^k_i, E^k_j$ of $E^k$. Then ${\mathcal G}_{k+1}$ is obtained from ${\mathcal G}_k$ by adding a new vertex  that represents the exceptional divisor $E^{k+1}_{k+1}$ and two new wedges connecting it with $E^{k+1}_i$ and $E^{k+1}_j$.
    \item The point $p_k$ belongs to three invariant compact components $E^k_i, E^k_j, E^k_\ell$ of $E^k$. Then ${\mathcal G}_{k+1}$ is obtained from ${\mathcal G}_k$ by adding a new vertex that represents the exceptional divisor $E^{k+1}_{k+1}$ and three new wedges connecting it with $E^{k+1}_i$, $E^{k+1}_j$ and $E^{k+1}_\ell$.
\end{enumerate}
A {\em chain } of length $s$ in ${\mathcal G}_N$ is any sequence
\begin{equation}
c=(E_{i_0},w_1, E_{i_1},w_2, E_{i_2},\ldots,w_{s-1},E_{i_{s-1}},w_s, E_{i_s})
\end{equation}
such that $w_{n}=E_{i_{n-1}}\cap E_{i_n}$ is a wedge for $n=1,2,\ldots,s$. If we have another chain
$$
c_1=(E_{i_s},w_{s+1}, E_{i_{s+1}},w_{s+2}, E_{i_{s+2}},\ldots,w_{t-1},E_{i_{t-1}},w_t, E_{i_t})
$$
starting at $E_{i_s}$, we can compose the two chains to obtain
$$
c*c_1=(E_{i_0},w_1, E_{i_1},w_2, E_{i_2},\ldots,w_{t-1},E_{i_{t-1}},w_t, E_{i_t}).
$$

Let us consider a complex number $\mu\ne 0$. The {\em transformed number } $c(\mu)$ of by the chain $c$ is defined as follows. If $s=0$ we put $c(\mu)=\mu$. Put
$$
c=c_{s-1}*(E_{i_{s-1}},w_s, E_{i_s})
$$
where $c_{s-1}$ has length $s-1$.
For
$
\alpha=\mbox{\rm Ind}({\mathcal F}, E_{i_{s}}; w_{s})
$,
we define $$c(\mu)=-\alpha c_{s-1}(\mu).$$
 Let us denote $c^{-1}$ the chain obtained by reversing the order in $c$. By  Remark \ref{rk:productoindices} we have that
\begin{equation}
\label{eq:reversible}
c^{-1}(c(\mu))=\mu;\; c(c^{-1}(\mu))=\mu.
\end{equation}

\begin{lemma}
\label{lema:cadenacircular}
 Consider a (circular) chain
\begin{eqnarray*}
c=(E_{i_0},w_1, E_{i_1},w_2, E_{i_2},\ldots,w_{s-1},E_{i_{s-1}},w_s, E_{i_s})
\end{eqnarray*}
such that $E_{i_0}=E_{i_s}$. For any $\mu\ne 0$ we have
$
c(\mu)=\mu.
$
\end{lemma}
\begin{proof} In view of Equation \ref{eq:reversible} the result is true if and only if it is true for one of the shifted chains
$$
c_j=(E_{i_j},w_j, E_{i_{j+1}},w_{j+1}, E_{i_{j+2}},\ldots,w_{s},E_{i_{s}}=E_{i_0},w_1, E_{i_1},\ldots,E_{i_{j-1}},w_{j-1}, E_{i_j}).
$$
We do induction on the number of vertices of the graph ${\mathcal G}_N$ and the length of $c$. If we have only one vertex, we are done. Let $v$ be the last vertex incorporated to the construction of ${\mathcal G}$. If this vertex $v$ does not appear in $c$, we are done by induction. Assume that $v$ appears in $c$.

If $v$ is a connected component of $\mathcal G$, we have only one vertex in $c$ and we are done. If $v$ is not isolated, we have three possibilities:
\begin{enumerate}
\item $v$ is connected with exactly one vertex $v_1$ with a wedge $w'_1$.
\item $v$ is connected with two vertices $v_1,v_2$ by means of respective wedges $w'_1,w'_2$. In this case $v_1$ and $v_2$ are connected by wedge $\tilde w_{12}$.
\item     $v$ is connected with three vertices $v_1,v_2,v_3$ by means of respective wedges $w'_1,w'_2,w'_3$. In this case $v_1,v_2,v_3$ are connected two by two by  wedges $\tilde w_{12},\tilde w_{13},\tilde w_{23}$.
\end{enumerate}

In case (1), up to performing a shift of $c$, we may assume that $c$ has the form
$$
c=(v,w'_1,v_1,w_2,\ldots,w_{s-1},v_1,w'_1,v)
$$
and we are done by induction applied to $c'=(v_1,w_2,\ldots,w_{s-1},v_1)$ as follows. Let $\alpha= \mbox{\rm Ind}({\mathcal F},v;w'_1)$, put
$
\mu'=-\alpha\mu
$, then we have
$$
c(\mu)=(-1/\alpha)c'(\mu')=(-1/\alpha)\mu'=\mu.
$$
In case (2), up to interchanging the role of $v_1$ and $v_2$ the appearance of $v$ may be in one of the following two forms,
\begin{eqnarray*}
c&=&c_1*(v_1,w'_1,v,w'_1,v_1)*c_2, \\
c&=& c_1*(v_1,w'_1,v,w'_2,v_2)*c_2.
\end{eqnarray*}
The first one is treated as in the previous case. Assume we have the second one. Let us denote
$$
\alpha=\mbox{\rm Ind}({\mathcal F}, v_1;\tilde w_{12});\;\beta=\mbox{\rm Ind}({\mathcal F}, v_1;\tilde w'_1);\;
\rho =\mbox{\rm Ind}({\mathcal F}, v_2;\tilde w'_1).
$$
We know that $\beta=-\alpha\rho$. Consider the circular chain
$$
\tilde c=c_1*(v_1,\tilde w_{12},v_2)*c_2.
$$
In view of the fact that
$$
(v_1,\tilde w_{12},v_2)(\tilde\mu)=-\alpha\tilde\mu=(-\beta)(-1/\rho)\tilde\mu= (v_1,w'_1,v,w'_2,v_2)(\tilde \mu),
$$
we deduce that $\tilde c(\mu)=c(\mu)$ and we are done since by induction we have $\tilde c(\mu)=\mu$.

Case (3) is treated as the previous one.
\end{proof}
Now we go to the proof of Proposition \ref{pro:indexpartialseparatriz}. Since the compact part of the partial separatrix $C$ is connected, we can join a generic point $p_1$ in $\Gamma_1$ with a generic point $p_2$ in $\Gamma_2$ by a real path $\gamma$. Moreover $\gamma$ may be chosen in such a way that it produces only finitely many changes of irreducible curves in $C$. The connected change of (trace) irreducible curves of $C$ gives a transition
 of invariant compact component of the divisor. In this way, we obtain a circular chain
$$
c=(E_{i_0},w_1, E_{i_1},w_2, E_{i_2},\ldots,w_{s-1},E_{i_{s-1}},w_s, E_{i_s}=E_{i_0})
$$
such that if $\mu$ is the index for $\Gamma_1$ then $c(\mu)$ is the index for $\Gamma_2$. By Lemma \ref{lema:cadenacircular} we have that $c(\mu)=\mu$ and the proof is ended.
\begin{remark} Corollary \ref{cor:transitioncompleta} may now be reformulated by stating that
$$
\left(\sum_{\gamma\in{\mathcal B}_iC}(\gamma,\Gamma)_p\right)\mbox{\rm Ind}(C;E_i)=-\alpha \left(\sum_{\eta\in{\mathcal B}_jC}(\eta,\Gamma)_p\right)\mbox{\rm Ind}(C;E_j).
$$
\end{remark}

\section{Real saddles at incomplete points}

Here we give a result relating incomplete points and real saddle curves. This is a  key point in the proof of Theorem \ref{teo:mainteo}.

\begin{proposition}
\label{pro:notallrealsaddles}
Let $p$ be an incomplete point belonging to a compact invariant component $E^k_i$.  There is a compact curve $\Gamma\subset\mbox{\rm Sing}{\mathcal F}_k$ with $\Gamma\subset E^k_i$ such that $\Gamma$ is not a real saddle.
\end{proposition}
\begin{proof} As usual we do induction on  $N-k$. If $k=N$ there is nothing to prove, since $p$ is a complete point. Assume that $k<N$. We assume without loss of generality that $p\in Y_k$. Moreover, since $p$ is an incomplete point, we necessarily have that $Y_k=\{p\}$ in view of Remark \ref{rk:completitudequirreduction}. Now, it is enough to find $\gamma\in {\mathcal B}^k_i(p)$ such that
$$
\mbox{\rm Ind}({\mathcal F},E_{i};\gamma)\notin {\mathbb R}_{<0}.
$$
 We assume by contradiction that all $\gamma\in {\mathcal B}^k_i(p)$ are real saddle curves.

{\em First case: $\pi_{k+1}$ is a dicritical blow-up}. We apply Proposition \ref{pro:indicesymultiplicidad} to the dicritical component $E^{k+1}_{k+1}$ to see that
\begin{equation}
\label{eq:dicincompleto}
\sum_{q\in L}\sum_{\gamma\in {\mathcal B}^k_i(p)}(\gamma',L)_q \mbox{\rm Ind}({\mathcal F},E_i;\gamma)=\sum_{q\in L}
\mbox{\rm Ind}({\mathcal F}\vert_{E^{k+1}_{k+1}},L;q)
\end{equation}
where $L=E^{k+1}_{k+1}\cap E^{k+1}_i$. The left hand side of Equation \ref{eq:dicincompleto} is a negative number but the right hand side coincides with the self-intersection of $L$ in $E^{k+1}_{k+1}$, that is, it has the value $+1$. This is the desired contradiction.

{\em Second case: $\pi_{k+1}$ is a non dicritical blow-up}. Put $L=E^{k+1}_{k+1}\cap E^{k+1}_i$ as before and
$
\alpha= \mbox{\rm Ind}({\mathcal F},E_i;L)
$. Let us consider a generic plane $\Delta$ at $p$ and ${\mathcal G}={\mathcal F}_k\vert_\Delta$. The blow-up $\pi_{k+1}$ induces a blow-up
$
\tilde\Delta\rightarrow \Delta
$
and the transform of $\mathcal G$ by this blow-up is $\tilde{\mathcal G}={\mathcal F}_{k+1}\vert_{\tilde\Delta}$. By the properties of the indices of Camacho-Sad we have
$$
\sum_{q\in \tilde\Delta\cap E^{k+1}_{k+1}} \mbox{\rm Ind}(\tilde{\mathcal G},\tilde\Delta\cap E^{k+1}_{k+1};q)=-1.
$$
Moreover, we have
$$
\sum_{q\in \tilde\Delta\cap E^{k+1}_{k+1}} \mbox{\rm Ind}(\tilde{\mathcal G},\tilde\Delta\cap E^{k+1}_{k+1};q)=
\mbox{\rm Ind}({\mathcal F}, E_{k+1};L)+
 \sum_{\Lambda\subset E^{k+1}_{k+1},\Lambda \ne L} d_{\Lambda}\mbox{\rm Ind}({\mathcal F},E_{k+1};\Lambda).
$$
We know that
$
\mbox{\rm Ind}({\mathcal F}, E_{k+1};L)= 1/\alpha
$
and by Proposition \ref{pro:indicesymultiplicidad} we have
$$
\sum_{\Lambda\subset E^{k+1}_{k+1},\Lambda \ne L} d_{\Lambda}\mbox{\rm Ind}({\mathcal F},E_{k+1};\Lambda)=
-\frac{1}{\alpha}\left\{
\sum_{q\in E^{k+1}_{k+1}}\sum_{\gamma\in {\mathcal B}^k_i(p)}(\gamma,L)_q \mbox{\rm Ind}({\mathcal F},E_i;\gamma)
\right\}.
$$
That is
$$
-1= \frac{1}{\alpha}-\frac{r}{\alpha}
$$
with $r<0$. Thus $\alpha=(r-1)$ is a negative real number and  $L$ is a real saddle. By induction hypothesis, all the points in $L$ must be complete, since otherwise the non real saddle in $E^{k+1}_i$ is not $L$ and projects to a non real saddle in $E^k_i$. Moreover, since the blow-up is non dicritical, there is at least one incomplete point $q\in E^{k+1}_{k+1}$. Let $\Theta$ be a curve through $q$ that is not a real saddle and consider a point $p'\in \Theta\cap L$. If $\Theta$ is a trace curve, by the transition of indices at complete points given in Corollary \ref{cor:transitioncompleta} we deduce the existence of a trace curve $\tilde\Theta\subset E^{k+1}_i$ such that  $\tilde\Theta\subset E^{k+1}_i$ is not a real saddle. If $\Theta$ is contained in the intersection of two divisors, by  Corollary \ref{cor:indicesymultiplicidad}  we find a non real saddle $\tilde\Theta\subset E^{k+1}_i$ with $\tilde\Theta\not\subset E^{k+1}_{k+1}$. The projection of $\tilde\Theta$ gives the desired contradiction.
\end{proof}

\section{Uninterrupted Nodal Components}
Let us recall the notion of uninterrupted nodal component introduced in \cite{Can-R-S}. By definiton, an {\em uninterrupted nodal component} of ${\mathcal F}_N, E^N$ is a connected union $\mathcal N$ of irreducible curves  $\Gamma\subset \mbox{\rm Sing}{\mathcal F}_N$ satisfying the following conditions:
\begin{enumerate}
\item Each $\Gamma\subset {\mathcal N}$ is a {\em nodal} curve (see Definition \ref{def:nodalsaddle}).
\item The component ${\mathcal N}$ is {\em uninterrupted} in the sense that there are exactly two curves $\Gamma_1$ and $\Gamma_2$ in ${\mathcal N}$ through any point $p\in {\mathcal N}$ of dimensional type three.
\end{enumerate}
Recall that an uninterrupted nodal component $\mathcal N$  is {\em incomplete} if and only if it intersects at least one compact dicritical component of the exceptional divisor $E^N$.
Otherwise, we say that $\mathcal N$ is {\em complete}.

The next result shows the compatibility between the uninterrupted nodal components and the partial separatrices, in the last step of the reduction of singularities.

\begin{proposition}[Global  trace transitions]
\label{pro:globaltransitions} Let $\mathcal N$ be a uninterrupted nodal component. Consider
 a partial separatrix $C$ and a compact invariant component $E_i$ of the exceptional divisor $E$.
If there is $\Gamma_0\in {\mathcal B}_iC$ with $\Gamma_0\subset {\mathcal N}$ then $\Gamma\subset {\mathcal N}$ for any $\Gamma\in {\mathcal B}_iC$.
\end{proposition}
\begin{proof}  The proof is similar to the proof of Proposition \ref{pro:indexpartialseparatriz}. Let us consider the dual graph ${\mathcal G}_N$ as in Proposition \ref{pro:indexpartialseparatriz}. Take two curves $\Gamma_0,\Gamma_1\in {\mathcal B}_iC$. We can connect $\Gamma_0,\Gamma_1$ by a circular chain
$$
c=(E_i=E_{i_0},w_1,E_{i_1},w_2,E_{i_2},\ldots,w_{s-1},E_{i_{s-1}},w_s,E_{i_s}=E_i)
$$
as in Lemma \ref{lema:cadenacircular}. Now, let us recall that at a point of dimensional type three we have either  no curves of ${\mathcal N}$ or exactly two of them. In this way, we have the following rule of behavior for the curves $\Gamma_{i_j}\subset E_{i_j}\cap C$ that we are considering in the chain $c$:
\begin{enumerate}
\item If $\Gamma_{i_{j-1}}\subset {\mathcal N}$ and $w_j\not\subset {\mathcal N}$, then $\Gamma_{i_{j}}\subset {\mathcal N}$.
\item If $\Gamma_{i_{j-1}}\subset {\mathcal N}$ and $w_j\subset {\mathcal N}$, then $\Gamma_{i_{j}}\not\subset {\mathcal N}$.
\item If $\Gamma_{i_{j-1}}\not\subset {\mathcal N}$ and $w_j\subset {\mathcal N}$, then $\Gamma_{i_{j}}\subset {\mathcal N}$.
\item If $\Gamma_{i_{j-1}}\not\subset {\mathcal N}$ and $w_j\not\subset {\mathcal N}$, then $\Gamma_{i_{j}}\not\subset {\mathcal N}$.
\end{enumerate}
Let us denote $\epsilon(w_{i_j})=-1$ if $w_{i_j}\subset {\mathcal N}$ and $\epsilon(w_{i_j})=1$ otherwise. Now, it is enough to prove that
$$
\epsilon(w_{i_1})\epsilon(w_{i_2})\cdots \epsilon(w_{i_s})=1.
$$
This can be done by the same arguments as in the proof of Lemma \ref{lema:cadenacircular}.
\end{proof}

Now, we consider an intermediate step $(M_k,F_k)$ of the reduction of singularities and we will study the transition properties
of a {\em fixed complete uninterrupted nodal component $\mathcal N$} at this level $k$ (see also \cite{Can-R-S}).
 We put  $${\mathcal N}_k=\rho_k({\mathcal N}).$$
Note that ${\mathcal N}_k\cap F_k$ is either a single point or a finite union of compact curves.
\begin{proposition}[Triple points transitions]
 \label{pro:tripetransition}
 Let $p\in F_k$ be  a point belonging to three compact components $E^k_i,E^k_j$ and $E^k_\ell$ of $E^k$. Assume that $\Gamma_\ell=E^k_i\cap E^k_j\subset {\mathcal N}_k$. Then $E^k_i,E^k_j$ and $E^k_\ell$ are invariant  and
$$
\Gamma_j\subset {\mathcal N}_k \Leftrightarrow \Gamma_i \not\subset {\mathcal N}_k,
$$
where $\Gamma_j=E^k_\ell\cap E^k_i$ and $\Gamma_i=E^k_\ell\cap E^k_j$.
\end{proposition}
\begin{proof} We do induction on  $N-k$. If $k=N$ we are done by the definition of complete uninterrupted nodal component.

Assume that $k<N$ and $p\in Y_k$ as usual. Since $p$ is in the intersection of three compact components then $Y_k=\{p\}$. Denote
$$
p'_u=E^{k+1}_{k+1}\cap \Gamma'_u,\quad  L'_u= E^{k+1}_{k+1}\cap E^{k+1}_u; \quad u\in\{i,j,\ell\}.
$$

By induction on $p'_\ell$ we have that $E^{k+1}_i, E^{k+1}_j$ and $E^{k+1}_{k+1}$ are invariant. In particular $\pi_{k+1}$ is non dicritical.  We also have that either $L'_i$ or $L'_j$ are contained in ${\mathcal N}_{k+1}$. Now by induction on $p'_j$ or $p'_i$ respectively, we deduce that $E^{k+1}_\ell$ and hence $E^{k}_\ell$ is invariant.

Assume now that $L'_i\subset {\mathcal N}_{k+1}$ and hence $L'_j\not \subset {\mathcal N}_{k+1}$. By induction on $p'_j$ we have two possibilities:
\begin{enumerate}
\item $\Gamma'_j\subset {\mathcal N}_{k+1}$ and $L'_\ell\not\subset {\mathcal N}_{k+1}$. It is not possible to have that $\Gamma'_i\subset {\mathcal N}_{k+1}$ since at $p'_i$ we have the two other corner curves not in ${\mathcal N}_{k+1}$.
\item $\Gamma'_j\not\subset {\mathcal N}_{k+1}$ and $L'_\ell\subset {\mathcal N}_{k+1}$. Then we have that $\Gamma'_i\subset {\mathcal N}_{k+1}$ since $L'_j\not\subset {\mathcal N}_{k+1}$.
\end{enumerate}
We conclude in the same way in the case that $L'_j\subset {\mathcal N}_{k+1}$ and  $L'_i\not \subset {\mathcal N}_{k+1}$.
\end{proof}

\begin{proposition}[Trace transitions]
\label{pro:tracetransitions} Let $C$ be a partial separatrix. Consider a point $p\in C_k$ complete for $C$ and belonging to a compact invariant component $E^k_i$.
Suppose that there is $\gamma\in {\mathcal B}_i^k(C;p)$ with $\gamma\subset {\mathcal N}_k$  and that $p$ belongs to another compact component $E^k_j$.
We have:
\begin{enumerate}
\item  $E^k_j$ is an invariant component.
\item  Put $\Gamma=E^k_i\cap E^k_j$. Then one of the following statements holds:
    \begin{enumerate}
    \item $\Gamma\subset {\mathcal N}_k$ and any $\tilde\gamma\in {\mathcal B}^k_j(C;p)$  is a real saddle.
        \item $\Gamma$ is a real saddle and for any $\tilde\gamma\in {\mathcal B}^k_j(C;p)$ we have $\tilde\gamma\subset {\mathcal N}_k$.
    \end{enumerate}
\end{enumerate}
\end{proposition}
\begin{proof} As usual, we do induction on $N-k$. If $k=N$, we are done. Indeed,
 $p$ does not belong to any dicritical compact component, since $p\in {\mathcal N}$ and  $\mathcal N$ is  complete. Moreover, the alternative in (2) means that $\mathcal N$ is uninterrupted.

Assume now that $k<N$ and $p\in Y_k$ as usual.
If $Y_k$ is a curve, there is only one compact component of $E^{k}$ through $p$ and $E^k_j$ does not exist. We assume thus that $Y_k=\{p\}$. Since $p\in C_k$ is a complete point for $C$, then $\pi_{k+1}$ is non dicritical.  Let us denote
$$
L'_i=E^{k+1}_{k+1}\cap E^{k+1}_{i},\; L'_j=E^{k+1}_{k+1}\cap E^{k+1}_{j}; \Gamma'= E^{k+1}_i\cap E^{k+1}_{j}
$$
and let $p'$ be the intersection point $E^{k+1}_{k+1}\cap \Gamma'$.

Let us see that $E^{k+1}_j$ is invariant.
 If $L'_i\subset{\mathcal N}$, then  $E^{k+1}_j$ is invariant by Proposition \ref{pro:tripetransition} applied at $p'$.
 If $L'_i\not\subset{\mathcal N}$, the strict transform $\gamma'$ of $\gamma$ intersects $L'_i$ at some points. Let $q$ be one of such points. By induction hypothesis on $q$ there is a curve $\gamma^*\subset E^{k+1}_{k+1}$ with $\gamma^*\subset {\mathcal N}_{k+1}$ corresponding to the same partial separatrix $C$. We conclude that $E^{k+1}_j$ is invariant by induction hypothesis applied at the points of $\gamma^*\cap E^{k+1}_j$.

 Now, assume that $\Gamma\subset {\mathcal N}_k$. Hence $\Gamma'\subset {\mathcal N}_{k+1}$. By Proposition \ref{pro:tripetransition}, we have two possible situations:
 \begin{enumerate}
 \item [i)] $L'_i\subset {\mathcal N}_{k+1}$ and $L'_j\not\subset {\mathcal N}_{k+1}$.
 \item [ii)] $L'_j\subset {\mathcal N}_{k+1}$ and $L'_i\not\subset {\mathcal N}_{k+1}$.
 \end{enumerate}
 Assume we have i). Consider a point $q\in\gamma'\cap L'_i$. By induction hypothesis at $q$, there is a curve $\gamma^*\in {\mathcal B}^{k+1}_{k+1}(C;q)$ such that $\gamma^*\not\subset {\mathcal N}_{k+1}$. We consider a point $q'\in \gamma^*\cap L'_j$. Also by induction hypothesis at $q'$ there is $\tilde\gamma'\in {\mathcal B}^{k+1}_j(C;q')$ such that $\tilde\gamma'\not\subset {\mathcal N}_{k+1}$. Now it is enough to take $\tilde\gamma$ the image of $\tilde\gamma'$ by $\pi_{k+1}$. In the case ii) we do a similar argumentation.

 Also, if
 $\Gamma\not\subset {\mathcal N}_k$ we have two possibilities:
 \begin{enumerate}
 \item [i)] $L'_i\subset {\mathcal N}_{k+1}$ and $L'_j\subset {\mathcal N}_{k+1}$.
 \item [ii)] $L'_j\not\subset {\mathcal N}_{k+1}$ and $L'_i\not\subset {\mathcal N}_{k+1}$.
 \end{enumerate}
 By the same kind of argumentations we find $\tilde\gamma\in {\mathcal B}^k_j(C;p)$ with $\tilde\gamma\subset {\mathcal N}_k$.

 The statements relative to the real saddles are consequence of Proposition \ref{pro:indicesymultiplicidad}.
\end{proof}
\section{Incompleteness of uninterrupted nodal components}
As explained in the Introduction,  Theorem \ref{teo:mainteo} is a consequence of the following result:
\begin{theorem}
 \label{teo:nodalcomponents}  Any uninterrupted
 nodal component $\mathcal N$ of ${\mathcal F}_N,E^N$ is incomplete.
\end{theorem}
In this section we provide a proof for Theorem \ref{teo:nodalcomponents}. We assume that ${\mathcal F}$ has no germ of invariant surface and that ${\mathcal N}$ is a {\em complete} uninterrupted nodal component. We shall find a contradiction with the fact that ${\mathcal N}$ is complete.

Let $b>0$ be the {\em date of birth of the compact part of }
$\mathcal N$, that is we assume that ${\mathcal N}_k\cap F_k$ is a single point for $0\leq k<b$ and that ${\mathcal N}_b$ contains at least one compact curve. Note that ${\mathcal N}$ contains  at least one compact curve, because $\pi_1$ is the blow-up centered at the origin and hence the fiber $F=\pi^{-1}(0)$ is the union of the compact components of $E$. If we take a point $q\in {\mathcal N}\cap F$, the compact components of $E$ through $q$ are invariant, by the completeness of $\mathcal N$. If the dimensional type of $q$ is two, the singular locus of ${\mathcal F}_N$ coincides locally with ${\mathcal N}$ and it is contained in the invariant components of $E$ through $p$. If the dimensional type is three, we have two curves of $\mathcal N$, one of them is necessarily contained in a compact invariant component. As a consequence of this we find that $1\leq b\leq N$.
\begin{lemma}
\label{lema:nomonoidal}
The blow-up $\pi_b$ is not centered at a germ of curve.
\end{lemma}
\begin{proof}
Suppose  by contradiction  that $\pi_b$ is the (monoidal) blow-up centered  at a germ of curve $(Y_{b-1},p)$. The point $p\in F_{b-1}$ is contained in a compact component $E^{b-1}_i$ of $E^b$ transversal to $Y_{b-1}$. Now, we have equireduction along $Y_{b-1}$. We consider all the blow-ups we do over $Y_{b-1}$ and we reach a desingularized situation over the point $p$. The fiber of $p$ contains a maximal connected union of compact curves in ${\mathcal N}$, say
$$
\Gamma_{j_1}\cup\Gamma_{j_2}\cup \cdots\cup \Gamma_{j_s}, \quad \Gamma_{j_\ell}\cap \Gamma_{j_{\ell+1}}\ne\emptyset, \; \ell=1,2,\ldots,s-1.
$$
Each $\Gamma_{j_\ell}$ is of the form  $\Gamma_{j_\ell}=E_i\cap E_{j_\ell}$ where $E_{j_\ell}$ is non compact. Moreover, by the fact that $\mathcal N$ is uninterrupted, we have two possibilities:
\begin{enumerate}
\item The curves $\Gamma_{j_\ell}$ represent all the components of $E$ contained in the inverse image of $Y_{b-1}$.
\item There are two noncompact curves $\gamma_1=E_{j_1}\cap E_{j_0}$ and $\gamma_s= E_{j_s}\cap E_{j_{s+1}}$ such that $\gamma_1,\gamma_s\subset{\mathcal N}$ and none of the curves $E_{j_\ell}\cap E_{j_{\ell+1}}$ are in ${\mathcal N}$ for $\ell=1,2,\ldots, s-1$.
\end{enumerate}
 Moreover the concerned divisors are non dicritical. Now, we can apply the refined Camacho-Sad Theorem \cite{Ort-B-V} to a transversal plane section at a generic point of $Y_{b-1}$ near $p$. In this way, we  find a non compact trace curve of generic index not in ${\mathbb R}_{>0}$ that cuts one of the compact curves $\Gamma_{j_\ell}$. Since ${\mathcal N}$ is uninterrupted there is a compact trace curve of ${\mathcal N}$ contained in $E_i^{b-1}$, this is the desired contradiction.
\end{proof}

In view of Lemma \ref{lema:nomonoidal} we suppose that $\pi_b$ is a (quadratic) blow-up  centered at the point $p$. We also know that $\pi_b$ is non dicritical, since there is a compact curve $\Gamma\subset E^b_b\cap {\mathcal N}_b$. Moreover, the point $p$ belongs only to compact components of $E^{b-1}$, otherwise the blow-up should be monoidal.

We  consider separately the cases $b=1$ and $b>1$.

Assume that $b=1$ and consider the exceptional divisor $E^1_1=\pi_1^{-1}(0)$.  The curve  $\Gamma\subset {\mathcal N}_1\cap E^1_1$ is a compact trace curve and thus there is a partial separatrix $C=C_\Gamma$ with $\Gamma\subset C_1$. Consider the set  $C_1\cap E^1_1$.  In view of Proposition \ref{pro:globaltransitions}, any irreducible component $\Gamma'$ of $C_1\cap E^1_1$ satisfies $\Gamma'\subset{\mathcal N}_1$. Recall that $C$ is an incomplete partial separatrix  in view of Proposition \ref{pro:incompletitudpartialsep} and thus there is at least a point $q\in C_1\cap E^1_1$ such that $q$ is incomplete for $C$ by Proposition \ref{pro:incompletodicriticoono}. Therefore, we have the following situation for $k=1$:
\begin{quote}
{\bf A($k$):} There is a compact trace curve $\Gamma\subset {\mathcal N}_k$ and a point $q\in \Gamma$ incomplete for the partial separatrix $C_\Gamma$.
\end{quote}

Assume now that $b>1$. In this case $p$ belongs to at least one compact component of $E^{b-1}$. Recall that all the components of $E^{b-1}$ through $p$ are compact. We discuss case by case.

(1). We have only one component $E^{b-1}_i$ of $E^{b-1}$ through $p$, which may be dicritical or invariant.

(1-a). Assume that $E^{b-1}_i$ is dicritical. We perform the blow-up $\pi_b$ and  obtain a trace curve \break $\Gamma\subset {\mathcal N}_b\cap E^{b}_b$. By Proposition \ref{pro:tracetransitions}, the points in $\Gamma\cap E^{b}_i$ are incomplete  for $C_\Gamma$. We arrive to situation {\bf A($k$)} for $k=b$.

(1-b). Let us suppose now that $E^{b-1}_i$ is  invariant.  Consider the projective line \break
$L=E^b_i\cap E^b_b$, then  either $L\subset {\mathcal N}_b$ or not.

 (1-b-1). Assume  that $L\subset {\mathcal N}_b$. By taking a generic plane section at $p$ and by  Camacho-Sad's argument on the sum of indices as in the proof of Proposition \ref{pro:notallrealsaddles}, we find a compact trace curve $\Theta\subset E^b_b$ such that the index of $\Theta$ is not in ${\mathbb R}_{> 0}$. We consider a point $q$ of intersection of $L$ and $\Theta$. If $q$ is a complete point for $C_\Theta$, by Proposition \ref{pro:tracetransitions} we should obtain a trace curve in $E^b_i$ contained in ${\mathcal N}_b$;  this is not possible since $b$ is the date of birth of $\mathcal N$. Thus $q$ is an incomplete point. We obtain the following situation for $k=b$:
\begin{quote}
{\bf B($k$):} There are a compact curve $\Gamma\subset {\mathcal N}_k$ such that $\Gamma=E^k_i\cap E^k_j$ is the intersection of two invariant compact components and an incomplete point $q\in \Gamma$.
\end{quote}

(1-b-2).
Assume that $L\not\subset {\mathcal N}_b$. Then there is a trace curve $\Gamma\subset {\mathcal N}_b\cap E^b_b$. Consider a point $q\in \Gamma\cap L$. If $q$ is complete for $C_\Gamma$, we apply the trace transitions of Proposition \ref{pro:tracetransitions} and this contradicts the fact that $b$ is the date of birth of $\mathcal N$. Thus the point $q$ is incomplete for $C_\Gamma$ and we arrive to situation {\bf A($k$)} for $k=b$.

(2). There are two components $E^{b-1}_i, E^{b-1}_j$ of $E^{b-1}$ through $p$.

(2-a). If both components are dicritical, we do an argument as in case (1-a) to obtain {\bf A($k$)} for $k=b$.

(2-b).  If $E^{b-1}_i$ in invariant and $E^{b-1}_j$ is dicritical, we have two possibilities:

(2-b-1). There is a trace curve $\Gamma\subset {\mathcal N}_b\cap E^b_b$.  Take a point $q\in \Gamma\cap E^b_j$, in view of Proposition \ref{pro:tracetransitions}, the point $q$  must be incomplete for $C_\Gamma$. We obtain {\bf A($k$)} for $k=b$.

(2-b-2). The other case is that $L=E^b_i\cap E^b_b$ is contained in $\mathcal N$. By Proposition \ref{pro:tripetransition} this is not possible. 

(2-c). Both $E^{b-1}_i$ and $E^{b-1}_j$ are invariant.
 
 (2-c-1). If $E^{b}_i\cap E^b_b\subset {\mathcal N}_b$,  by Proposition \ref{pro:tripetransition} we have that $E^{b}_j\cap E^b_b\subset {\mathcal N}_b$. By the already used argument on the sum of indices for a generic plane section at $p$, we find  a trace curve $\Theta\subset E^b_b$ of index not in ${\mathbb R}_{>0}$. The trace transitions of $\Theta$ described in Proposition \ref{pro:tracetransitions} will produce  curves of ${\mathcal N}_{b-1}$ previously existing in $E^{b-1}_i\cup E^{b-1}_j$, unless we have incomplete points in $E^{b}_i\cap E^{b}_b$ and $E^{b}_j\cap E^{b}_b$ at the intersections with $\Theta$. Thus we obtain {\bf B($k$)} for $k=b$. 
 
 (2-c-2). Assume now that $E^{b}_i\cap E^b_b\not\subset {\mathcal N}_b$. By Proposition \ref{pro:tripetransition} we also have that $E^{b}_j\cap E^b_b\not\subset {\mathcal N}_b$, since otherwise we should have $E^{b-1}_i\cap E^{b-1}_j\subset {\mathcal N}_{b-1}$. The other possible curves in $E^b_b$ are of trace type and thus  the curve $\Gamma\subset {\mathcal N}_b$ that appears after $\pi_b$ is a trace curve. We obtain  {\bf A($k$)} for $k=b$ as in (1-b-2).

(3). There are three components $E^{b-1}_i, E^{b-1}_j$ and $E^{b-1}_\ell$ of $E^{b-1}$ through $p$. We use the same kind of argumentation as in the cases (1) and (2) to reach one of the situations {\bf A($k$)} or {\bf B($k$)} for $k=b$.
\\

We have identified two situations {\bf A($k$)} and {\bf B($k$)} such that one of them appears in the birth level of the uninterrupted nodal component $\mathcal N$. We would like to show the persistency of this phenomenon at further levels of the reduction of singularities. However, another situation must be considered, which is the following:
\begin{quote}
{\bf C($k$):} There is a compact invariant component $E^{k}_i$, a compact curve $\Gamma\subset E^k_i\cap {\mathcal N}_k$ and an incomplete point $q\in \Gamma$ such that the following holds:  every global irreducible curve $\Theta\subset\mbox{\rm Sing}{\mathcal F}_k\cap E^k_i$  with $q\in \Theta$  is either in ${\mathcal N}_k$ or a real saddle.
\end{quote}

\begin{proposition}[Persistency]
\label{prop:persistency}
Assume that there is an index $1\leq k<N$, a global curve $\Gamma\subset {\mathcal N}_k$ and an incomplete point $q\in \Gamma$ in one of the situations {\bf A($k$)},  {\bf B($k$)} or  {\bf C($k$)}. Then there is a global curve $\Gamma'\subset {\mathcal N}_{k+1}$ and an incomplete point $q'\in \Gamma'$ in one of the situations {\bf A($k+1$)},  {\bf B($k+1$)} or  {\bf C($k+1$)}.
\end{proposition}
\begin{proof}  If $\pi_{k+1}$ is centered at $Y_k$ with $q\notin Y_k$, we obviously reach  {\bf A($k+1$)}, {\bf B($k+1$)} or {\bf C($k+1$)} at the ``same'' point $q$. Thus, we assume  $q\in Y_k$. Moreover, since $q$ is incomplete, we have $Y_k=\{q\}$. Indeed, if $Y_k$ is a germ of curve, the point $q$ is complete. Then $E^{k+1}_{k+1}=\pi_k^{-1}(q)$ is a projective plane.

(a). Assume that we have {\bf A($k$)}. Let $E^k_i$ be the compact invariant component such that $\Gamma\subset E^k_i$. We consider two cases:

(a-1). {\em The blow-up $\pi_{k+1}$ is dicritical}. We consider the strict transform $\Gamma'$ of $\Gamma$ and a point \break $q'\in \Gamma'\cap E^{k+1}_{k+1}$. In view of Proposition \ref{pro:tracetransitions} the point $q'$ must be incomplete for $C_{\Gamma}$ and we recover the situation {\bf A($k+1$)}.

(a-2). {\em  The blow-up $\pi_{k+1}$ is non dicritical}.  Let us put $L=E^{k+1}_{k+1}\cap E^{k+1}_i$. 

(a-2-1). Assume first that $L\subset {\mathcal N}_{k+1}$. If there is an incomplete point $q'\in L$ we obtain {\bf B($k+1$)}. Thus we  assume that all the points in $L$ are complete. We find an incomplete point $q'\in E^{k+1}_{k+1}$. By Proposition \ref{pro:notallrealsaddles}, there is a global irreducible curve $\Gamma'\subset {E^{k+1}_{k+1}}$ with $q'\in \Gamma'$ that is not a real saddle. We consider a (complete) point $p'\in L\cap \Gamma'$. Now, by Proposition \ref{pro:tracetransitions} or Proposition \ref{pro:tripetransition} we see that $\Gamma'$ must be contained in ${\mathcal N}_{k+1}$.
This argument also works for all non real saddle curves through $q'$. Hence we find {\bf C($k+1$)} or {\bf B($k+1$)} at $q'$.

(a-2-2).
It remains to consider the case that $L\not\subset {\mathcal N}_{k+1}$. If there is a point $q'\in L\cap \tilde \Gamma$  incomplete  for $C_\Gamma$, where $\tilde\Gamma$ is the strict transform of $\Gamma$, we obtain {\bf A($k+1$)} at $q'$. If not, we consider the transitions given in Proposition \ref{pro:tracetransitions} to see that $C_\Gamma\cap E^{k+1}_{k+1}$ is contained in ${\mathcal N}_{k+1}$. Moreover, there exists a point $q'\in C_\Gamma\cap E^{k+1}_{k+1}$ incomplete for  $C_\Gamma$. We recover {\bf A($k+1$)} at $q'$.

(b). Assume we have {\bf B($k$)}. Put $L_i=E^{k+1}_{k+1}\cap E^{k+1}_i$ and $L_j=E^{k+1}_{k+1}\cap E^{k+1}_j$. Let $p'=L_i\cap L_j$. By Proposition \ref{pro:tripetransition} we know that $\pi_{k+1}$ is a non dicritical blow-up. Moreover, we have that  $$L_i\subset{\mathcal N}_{k+1} \Leftrightarrow  L_j\not\subset{\mathcal N}_{k+1}.$$
To fix ideas, suppose that $L_i\subset{\mathcal N}_{k+1}$ and $L_j\not\subset{\mathcal N}_{k+1}$. If there is an incomplete point at $L_i$ we have {\bf B($k+1$)} at such a point. So we assume that all the points in $L_i$ are complete. This means that there is an incomplete point $q'\in E^{k+1}_{k+1}\setminus L_i$. We repeat at this point the previous argument for the case (a-2-1)
and we recover {\bf C($k+1$)} at $q'$.

(c). Let us assume finally that we have {\bf C($k$)}. We also suppose that we are not in the situations {\bf A($k$)} or {\bf B($k$)} already studied and hence $\Gamma$ is a trace curve and $q$ is complete for $C_\Gamma$.
As in case (b), we may assume that $\pi_{k+1}$ is non dicritical, otherwise we obtain {\bf A($k+1$)} at the strict transform of $\Gamma$.  Let us put $L=E^{k+1}_{k+1}\cap E^{k+1}_i$. 

(c-1). Assume first that $L\subset {\mathcal N}_{k+1}$. We may assume that all the points in $L$ are complete, otherwise we get {\bf B($k+1$)}. All the global irreducible curves $\Theta\subset E^{k+1}_{k+1}$ with $\Theta \ne L$ are either real saddles or curves in ${\mathcal N}_{k+1}$ in view of Propositions \ref{pro:tracetransitions} and \ref{pro:tripetransition}. On the other hand, we necessarily have an incomplete point $q'\in E^{k+1}_{k+1}$.  The non real saddle passing through $q'$ given  by Proposition \ref{pro:notallrealsaddles} is then contained in ${\mathcal N}_{k+1}$, as well as any other non real saddle curve. Thus, we recover {\bf C($k+1$)} at $q'$.

(c-2). Let us assume that $L\not\subset {\mathcal N}_{k+1}$. Let $\Gamma'\subset{\mathcal N}_{k+1}$ be the strict transform of $\Gamma$ and take a complete point $p\in \Gamma'\cap L$. By the transition rules in Proposition \ref{pro:tracetransitions} we obtain that $L$ is a real saddle. If there is an incomplete point $q'\in L$ we are done, since it satisfies {\bf A($k+1$)}. We suppose that all the points in $L$ are complete and we take an incomplete point $q'\in E^{k+1}_{k+1}\setminus L$. Let us see that all the global irreducible curves $\Theta\subset E^{k+1}_{k+1}\cap \mbox{\rm Sing}{\mathcal F}_{k+1}$ are either real saddles or contained in ${\mathcal N}_{k+1}$. In this way we obtain {\bf C($k+1$)} at $q'$ and we are done.
  We look at the transitions through $L$ at $\Theta\cap L$ described in Corollary \ref{cor:transitioncompleta}. Recalling that the curves in $E^{k+1}_{i}\cap \mbox{\rm Sing}{\mathcal F}_{k+1}$ arriving at $L$ are either real saddle curves  or  in ${\mathcal N}_{k+1}$, we see that $\Theta$ is also in ${\mathcal N}_{k+1}$ or a real saddle curve. 
\end{proof}
As a consequence of Proposition \ref{prop:persistency} we arrive to {\bf A,B} or {\bf C} in the final step, which is not possible since all the points in the final step are complete points. This is the desired contradiction. Thus, the only possibility is that there are no complete uninterrupted nodal components. This ends the proof of Theorem \ref{teo:nodalcomponents}.

\end{document}